\documentclass[10pt,twoside, a4paper, english, reqno]{amsart}
\usepackage{listings, graphicx, amsmath, varioref, amscd, amssymb,color, bm}
\usepackage[pdftex, colorlinks=true, backref]{hyperref}

\numberwithin{equation}{section}
\allowdisplaybreaks

\newtheorem{theorem}{Theorem}[section]
\newtheorem{lemma}[theorem]{Lemma}

\theoremstyle{definition}
\newtheorem{definition}[theorem]{Definition}

\theoremstyle{remark}
\newtheorem{remark}[theorem]{Remark}

\newcommand{\Div}[1]{\operatorname{div} #1}
\newcommand{\Rot}[1]{\operatorname{curl} #1}

\newcommand{\Curl}[1]{\operatorname{curl}#1}

\newcommand{\vc}[1]{{\bm{#1}}}

\newcommand{\weak}{\rightharpoonup}

\newcommand{\norm}[1]{\left\Vert#1\right\Vert}
\newcommand{\abs}[1]{\left|#1\right|}
\newcommand{\Set}[1]{\left\{#1\right\}}
\newcommand{\jump}[1]{\left[#1\right]}
\newcommand{\vrho}{\varrho}
\newcommand{\inb}{\in_{\text{b}}}
\newcommand{\Dt}{\Delta t}
\newcommand{\R}{\mathbb{R}}
\newcommand{\weakto}{\rightharpoonup}
\newcommand{\mcw}{\mathcal{W}}

\newcommand{\Om}{\ensuremath{\Omega}}
\newcommand{\cOm}{\ensuremath{\overline{\Omega}}}
\newcommand{\pOm}{\ensuremath{\partial\Omega}}
\newcommand{\Dom}{(0,T)\times\Omega}
\newcommand{\cDom}{[0,T)\times\overline{\Omega}}

\newcommand{\eff}{P_{\text{eff}}}
\newcommand{\binner}{\partial E \setminus \partial \Om}

\newcommand{\solutiontext}
{
	Let $\Set{(\vrho_{h},\vc{w}_{h},\vc{u}_{h})}_{h>0}$ 
	be a sequence of numerical solutions 
	constructed according to \eqref{eq:num-scheme-II} 
	and Definition \ref {def:num-scheme}.
}

\title[Mixed FEM for a semi--stationary Stokes system]
{Convergence of a mixed method 
for a semi-stationary compressible Stokes system}

\author[K. H. Karlsen]{Kenneth H. Karlsen}
\address[Kenneth H. Karlsen]{\newline
         Centre of Mathematics for Applications \newline
         University of Oslo\newline
         P.O. Box 1053, Blindern\newline
         N--0316 Oslo, Norway\newline
         and\newline
		 Center for Biomedical Computing,\newline
         Simula Research Laboratory\newline
         P.O. Box 134\newline
         N--1325 Lysaker, Norway}
\email[]{kennethk@math.uio.no}
\urladdr{http://folk.uio.no/kennethk}

\author[T. K. Karper]{Trygve K. Karper} 
\address[Trygve K. Karper]{\newline 
		Centre of Mathematics for Applications \newline 
		University of Oslo\newline 
		P.O. Box 1053, Blindern\newline 
		N--0316 Oslo, Norway}
\email[]{t.k.karper@cma.uio.no}
\urladdr{http://folk.uio.no/trygvekk/}

\date{\today}

\subjclass[2000]{Primary 35Q30, 74S05; Secondary 65M12}

\keywords{Semi--stationary Stokes system, compressible fluid flow, Navier-slip boundary condition, 
mixed finite element method, discontinuous Galerkin scheme, convergence}

\thanks{This work was supported by the Research Council of Norway through
an Outstanding Young Investigators Award (K. H. Karlsen). 
This article was written as part of the the international research program
on Nonlinear Partial Differential Equations at the Centre for
Advanced Study at the Norwegian Academy of Science
and Letters in Oslo during the academic year 2008--09.}

\begin{document}

\begin{abstract}
We propose and analyze a finite element method for a semi--stationary 
Stokes system modeling compressible fluid flow 
subject to a Navier--slip boundary condition. 
The velocity (momentum) equation is approximated by a mixed finite element 
method using the lowest order N{\'e}d{\'e}lec spaces of the first kind. 
The continuity equation is approximated by a 
standard piecewise constant upwind discontinuous 
Galerkin scheme. Our main result states that the 
numerical method converges to a weak solution. 
The convergence proof consists of two main steps: 
(i) To establish strong spatial compactness of the 
velocity field, which is intricate since the element spaces are 
only $\Div{}$ or $\Curl{}$ conforming. 
(ii) To prove that the discontinuous Galerkin approximations 
converge strongly, which is required in view of the nonlinear pressure function. 
Tools involved in the analysis include a higher integrability 
estimate for the discontinuous Galerkin approximations, a discrete equation for 
the effective viscous flux, and various renormalized formulations 
of the discontinuous Galerkin scheme.  
\end{abstract}

\maketitle

{\small \tableofcontents}

\section{Introduction}\label{sec:intro}
The purpose of this paper is to prove convergence of a 
finite element method for the semi--stationary 
barotropic compressible Stokes system
\begin{align}
	\vrho_{t} + \Div (\vrho\vc{u}) 
	&= 0, \quad \text{in $(0,T) \times \Omega$}, \label{eq:contequation}\\
	-\mu \Delta \vc{u} - \lambda D\Div\vc{u} + Dp(\vrho) &= \vc{f}, 
	\quad \text{in $(0,T)\times \Omega$},
	\label{eq:momentumeq} 
\end{align}
with initial data
\begin{align}\label{initial}
	\vrho|_{t=0} & = \vrho_{0},
	\quad \textrm{on $\Omega$}.
\end{align}
Here $\Omega$ is a simply connected, bounded, open, 
polygonal domain in $\R^N$ ($N=2,3$), with Lipschitz 
boundary $\partial \Omega$, and $T>0$ is a fixed final time. 
The unknowns are the density $\vrho = \vrho(t,\vc{x}) \geq 0$ 
and the velocity $\vc{u} = \vc{u}(t,\vc{x}) \in \R^N$, with 
$\vc{x} \in \Omega$ and $t\in (0,T)$. 
We denote by $\Div$ and $D$ the usual spatial divergence 
and gradient operators and by $\Delta$ the spatial Laplace operator.

The pressure function is assumed to be of 
the form $p(\vrho) = a\vrho^\gamma$, with $a>0$ (Boyle's law). 
Typical values of $\gamma$ ranges from a maximum of $\frac{5}{3}$ 
for monoatomic gases, through $\frac{7}{5}$ for diatomic gases 
\emph{including air}, to lower values close to $1$ for 
polyatomic gases at high temperatures. Throughout this paper we will 
always assume that $\gamma>1$. The case $\gamma=1$ can also be 
treated; indeed, it is simpler since the pressure function is linear. 
Furthermore, the viscosity coefficients $\mu, \lambda$ 
are assumed to be constant and satisfy $\mu> 0, N\lambda+2\mu\geq 0$.

The study of the system \eqref{eq:contequation}--\eqref{eq:momentumeq} 
can be motivated in several ways. Firstly, the system can be used 
as a model equation for the barotropic compressible Navier--Stokes equations.  
This might be a reasonable approximation for strongly viscous fluids, where 
convection may be neglected. Secondly, Lions \cite{Lions:1998ga} use solutions of 
\eqref{eq:contequation}--\eqref{eq:momentumeq}
to construct solutions to the barotropic compressible Navier--Stokes equations. 

Among many others, the semi--stationary system \eqref{eq:contequation}--\eqref{initial} 
has been studied by Lions in \cite[Section 8.2]{Lions:1998ga}. 
He proves the existence of weak solutions and some 
higher regularity results. In particular, weak solutions was proven to be
unique in the case of periodic boundary conditions or when 
the equations are solved on the hole of $\mathbb{R}^N$.  
Uniqueness was not obtained in the case of regular Dirichlet boundary conditions 
and moreover higher regularity results was only shown to hold locally.

In this paper we impose the following boundary conditions, which are 
relevant in the context of geophysical fluids and shallow water models:
\begin{equation}\label{eq:bc-normal}
	\vc{u} \cdot \nu = 0, 
	\quad \textrm{on $(0,T)\times \partial \Omega$},
\end{equation}
and
\begin{equation}\label{eq:bc-navierslip}
	\begin{split} 
		\Curl \vc{u} & = 0, \quad 
		\textrm{on $(0,T)\times \partial \Omega$ if $N=2$}, \\
		\Curl \vc{u}\times \nu  &= 0, \quad 
		\textrm{on $(0,T)\times\partial \Omega$ if $N=3$},
	\end{split}
\end{equation}
where $\nu$ denotes the unit outward normal to $\partial \Omega$.
The first condition is a natural condition of impermeability type on the normal velocity.
The second condition is in the literature often referred to as the Navier--slip condition.  
It can be interpreted as a viscous dissipation term at the boundary 
(more precisely ``non-dissipation" since this term 
is equal to zero) \cite{Lions:1998ga}. 

In some geophysical applications, conditions 
like \eqref{eq:bc-normal}--\eqref{eq:bc-navierslip} 
are preferred over the classical Dirichlet condition since the latter 
necessitates expensive calculations of boundary layers. 
Of more importance to this paper, the boundary conditions 
\eqref{eq:bc-normal}--\eqref{eq:bc-navierslip} 
will allow us to use the finite element method 
in a solution space that can be split into two orthogonal 
parts in terms of a discrete version of the Hodge decomposition, a 
fact that will play a crucial role in our analysis.

Although many numerical methods have been proposed for the compressible 
Stokes and Navier--Stokes equations, the convergence 
properties of these methods are mostly unsettled, especially in 
several spatial dimensions. Ultimately, it is not clear if these 
numerical methods converge to a weak solution as the discretization parameters tend to zero. 
In one dimension, the available results are due 
to Hoff and his collaborators \cite{Zarnowski:1991uq, Zhao:1994fk, Zhao:1997qy}. 
All these results apply to the compressible Navier--Stokes equations 
in Lagrangian coordinates, and moreover require the initial density to be of bounded variation.  
Interesting results regarding the existence and long time behavior of solutions to the 
one dimensional compressible Navier--Stokes have also been obtained using
semi--discrete finite difference schemes in \cite{Hoff:1987fj,Hoff:1991kx,Chen:2000yq}, again 
in Lagrangian coordinates with the initial density of bounded total variation. 
In more than one spatial dimension, we refer to a recent paper
\cite{Gallouet:2007lr} in which a convergent numerical method for a stationary compressible Stokes system 
is proposed. The Stokes system considered in \cite{Gallouet:2007lr}
is similar to \eqref{eq:contequation}--\eqref{eq:momentumeq}
with linear pressure and no temporal dependence.

Let us now discuss our choice of numerical method for the semi-stationary Stokes system. 
For the discretization of \eqref{eq:contequation} we utilize a 
discontinuous Galerkin scheme based on piecewise constant approximations in space and time.
The discontinuous Galerkin scheme was introduced more than 30 
years ago \cite{Lasaint:1974kh,Reed:1973hq} and has since 
then undergone a blooming development, cf.~\cite{Cockburn:1999ud,Cockburn:2001bd} 
for a review. In the context of linear transport equations 
with rough (i.e., non-Lipschitz) coefficients, a discontinuous 
Galerkin scheme, with piecewise polynomial approximations 
of arbitrary degree in the spatial variable and piecewise constant or linear approximations 
in the temporal variable, has recently been analyzed by Walkington in \cite{Walkington:2005jl}. 
The work \cite{Walkington:2005jl} is further developed in \cite{Liu:2007fe} for 
the variable-density incompressible Navier-Stokes equations.

Let us now turn to the velocity (or momentum) equation \eqref{eq:momentumeq}. 
By introducing the vorticity $\vc{w}= \Curl \vc{u}$ as an auxiliary 
unknown, keeping in mind the vector identity $-\Delta = \Curl\Curl -D \Div$, 
we can recast the momentum equation as
\begin{equation}\label{eq:vorticity-form}
	\mu \Curl \vc{w} - (\lambda + \mu) D \Div \vc{u}
	+Dp(\vrho) = \vc{f},
\end{equation}
where we suppress the time variable $t$ (we refer the reader to subsequent 
sections for more precision).  Hence the velocity equation \eqref{eq:momentumeq}, 
together with the boundary conditions 
\eqref{eq:bc-normal}--\eqref{eq:bc-navierslip}, admits a formulation 
that lends itself naturally to a mixed finite element 
method \cite{Girault:1986fu,Nedelec:1980ec,Nedelec:1986bv}.  

Denote by $\vc{W}^{\Div, 2}_{0}$ the vector fields $\vc{u}$ on $\Omega$ 
for which $\Div \vc{u}\in L^2$ and $\vc{u} \cdot \nu|_{\partial \Omega} = 0$, and by 
$\vc{W}_{0}^{\Curl, 2}$ the vector fields $\vc{w}$ on $\Omega$ for which 
$\Curl \vc{w}\in L^2$ and $\vc{w} \times \nu|_{\partial \Omega}=0$. 
We choose corresponding mixed finite element spaces $\vc{V}_{h} \subset \vc{W}^{\Div,2}_{0}$ 
and $\vc{W}_{h}\subset \vc{W}^{\Curl,2}_{0}$ based on 
N{\'e}d{\'e}lec's elements of the first kind \cite{Nedelec:1980ec}. 
The mixed finite element method seeks functions $\vc{w}_h\in \vc{W}_{h}$ 
and $\vc{u}_h\in \vc{V}_{h}$ such that
\begin{equation*}
	\begin{split}
		& \int_{\Omega}\mu\Curl \vc{w}_h \vc{v}_h
		+\left[(\mu + \lambda)\Div \vc{u}_h 
		- p(\vrho_h)\right]\Div \vc{v}_h\ dx
		= \int_{\Omega}\vc{f}_h\vc{v_h}\ dx, \\
		& \int_{\Omega} \vc{w}_h\vc{\eta}_h 
		-\Curl \vc{\eta}_h \vc{u}_h \ dx = 0, 
	\end{split}
\end{equation*}
for all $(\vc{\eta}_h,\vc{v}_h) \in \vc{W}_{h}\times \vc{V}_{h}$, where 
$\vrho_h,\vc{f}_h$ are given piecewise constant functions. 

Let us denote the numerical solution of the 
semi-stationary Stokes system by $(\vrho_{h},\vc{w}_{h},\vc{u}_{h})=(\vrho_{h},\vc{w}_{h},\vc{u}_{h})(t,\vc{x})$. 
The main goal is to prove that $\Set{(\vrho_{h},\vc{w}_{h},\vc{u}_{h})}_{h>0}$ 
converges to a weak solution, at least along a subsequence. 
The challenging issue is to show that the 
density approximations $\vrho_h$, which on the outset is only 
weakly compact in $L^2$, in fact converges strongly. Strong convergence 
is mandatory if we want to recover the semi-stationary Stokes system when 
taking the limit in the discrete equations as $h\to 0$. Related to this issue, the above 
mixed method enjoys some advantages over the traditional finite 
element method based on $\vc{H}^1$ elements. 
In particular, the approximation spaces $\vc{W}_{h}$ and $\vc{V}_{h}$ satisfy
$$
\vc{V}_{h} = \Curl \vc{W}_{h} + \vc{Z}_{h},
$$
for some $\vc{Z}_{h} \subset \vc{V}_{h}$ satisfying $\vc{Z}_{h} \perp \Curl \vc{W}_{h}$.  
An immediate consequence of this discrete Hodge decomposition 
is that upon writing $\vc{u}_{h} = \Curl \vc{\eta}_{h} + \vc{z}_{h}$, 
we see that \textit{only} $\vc{z}_{h}$ is coupled 
to the density $\vrho_{h}$ and moreover that $\Curl \vc{w}_{h}$, and hence 
$\vc{w}_{h}$, only depends on the data $\vc{f}$. 
More importantly, equipped with the discrete Hodge decomposition, we 
can separate the quantity $\eff(\vrho_{h},\vc{u}_{h})=
P(\vrho_{h}) - (\lambda+\mu)\Div \vc{u}_{h}$ 
from the vorticity. The quanity $\eff(\vrho_{h},\vc{u}_{h})$ is the so-called 
effective viscous flux \cite{Lions:1998ga} associated with our discrete equations.
The fact that we can separate the effective viscous flux from the vorticity makes 
it possible to prove the following weak continuity property:
\begin{equation}\label{eq:intro-weakcont}
	\lim_{h \rightarrow 0}\iint \eff(\vrho_{h},\vc{u}_{h})\,\vrho_{h}\ dxdt 
	=\iint \overline{\eff}\, \vrho \ dxdt
	\quad \text{($\overline{\eff}, \vrho$ are weak $L^2$ limits),}
\end{equation}
which is the decisive ingredient in the proof of strong convergence 
of the density approximations $\vrho_{h}$. Related to \eqref{eq:intro-weakcont}, we prove 
a higher integrability estimate on the pressure ensuring 
that $p(\vrho_{h})$, and thus also $\eff(\vrho_{h},\vc{u}_{h})$, is weakly compact in $L^2$. 
The energy estimate only provides a uniform bound on $p(\vrho_{h})$ in 
$L^\infty(L^1)$, so a priori it is not even clear that $p(\vrho_{h})$ 
converges weakly to an integrable function. Our strong convergence argument is inspired by the 
work of Lions on the compressible Navier-Stokes 
equations, cf.~\cite{Lions:1998ga}.

As part of the analysis, we also show that $\vrho_{h}\vc{u}_{h}$ 
converges weakly to $\vrho \vc{u}$, where $\vrho$ and $\vc{u}$ 
are weak limits of $\vrho_{h}$ and $\vc{u}_{h}$, respectively. 
This convergence is not immediate since the element spaces utilized for 
the velocity approximations are merely div or curl conforming. 
In view of the discrete continuity equation (discontinuous 
Galerkin scheme), we easily obtain a bound on 
$(\vrho_{h})_t$ in, say, $L^1(W^{-1,1})$. To conclude we need 
a spatial translation estimate of the form
\begin{equation}\label{intro:eq3}
	\norm{\vc{u}_{h}-\vc{u}_{h}(\cdot,\cdot + \xi)}_{L^{2}(\vc{L}^2)} 
	\to 0 \quad \text{as $\abs{\xi} \to 0$, uniformly in $h$.}
\end{equation}
In view of the discrete Hodge decomposition, we will actually only need 
\eqref{intro:eq3} for weakly curl 
free approximations with a $\vc{L}^2$ bounded divergence.

For velocity fields that are independent of time $t$, \eqref{intro:eq3} implies the 
$\vc{L}^2$ compactness of $\Set{\vc{u}_{h}}_{h>0}$. In the time independent case, 
it is known that weakly curl free approximations with $\vc{L}^2$ bounded divergence is compact in $\vc{L}^2$
provided the approximation spaces satisfy the commuting diagram property \cite{Snorre}. 
However, despite the fact that the element spaces used here satisfy this 
property, the inclusion of time in $\vc{u}_{h}(t,x)$ makes earlier results inadequate. 
Specifically, to apply known result we would need $L^\infty$ control in time 
of the velocity approximations. Unfortunately, this is not available in general for our problem. 
As a consequence, we shall provide a direct argument 
for the spatial translation estimate \eqref{intro:eq3}.

We wish to point out that although the boundary conditions \eqref{eq:bc-normal}--\eqref{eq:bc-navierslip} are 
not covered by Lions' results \cite{Lions:1998ga}, his proofs can be adapted 
to yield existence, uniqueness, and regularity results for \eqref{eq:contequation}--\eqref{eq:momentumeq} 
with the boundary condtions \eqref{eq:bc-normal}--\eqref{eq:bc-navierslip}. We will 
not pursue this project here, except for the existence part, which will be an immediate 
consequence of our convergence result. However, let us 
remark that the Navier--slip condition \eqref{eq:bc-navierslip} is technically
easier to handle than a Dirichlet condition, both from a mathematical and 
numerical point of view.  The primary reason for this lies in 
the need for solutions of the auxiliary problem
\begin{equation}\label{divvf}
	\Div \vc{v} = \vc{f}, 
	\quad \Curl \vc{v} = 0.
\end{equation}
If $\int_{\Omega}\vc{f}\ dx=0$,  the function $\vc{v}$ 
will satisfy the boundary conditions \eqref{eq:bc-normal}--\eqref{eq:bc-navierslip}. 
In other situations, like periodic boundary conditions or
when the equations are solved on $\mathbb{R}^N$, the boundary 
values of $\vc{v}$ does not matter. However, it is evident 
that $\vc{v}$ cannot be required both to satisfy 
Dirichlet boundary conditions and \eqref{divvf}. 
Thus, \eqref{divvf} can only be required to hold locally whenever Dirichlet 
boundary conditions are imposed. To avoid ``localizing" various 
discrete arguments, which sometimes 
can require elaborate work, we have chosen to consider the 
Navier--slip type condition \eqref{eq:bc-navierslip} 
instead of the no--slip Dirichlet condition.

This paper is organized as follows: In Section \ref{sec:prelim}, we introduce notation and 
list some basic results needed for the later analysis. 
Moreover, we recall the usual notion of weak solution 
and introduce a mixed weak formulation of the velocity equation. 
Finally, we introduce the finite element spaces and review some of their basic properties. 
In Section \ref{sec:numerical-method}, we present the numerical method and state our main 
convergence result. The existence of a solution to the discrete equations is confirmed 
in Section \ref{sec:existence-method}. Section \ref{sec:basic-est} is 
devoted to deriving basic estimates. 
In Section \ref{sec:conv}, we prove the main convergence 
result stated in Section \ref{sec:numerical-method}. 
The proof is divided into several steps (subsections), including 
convergence of the continuity scheme, weak continuity 
of the discrete viscous flux, strong convergence 
of the density approximations, and convergence of the velocity scheme.

\section{Preliminary material}\label{sec:prelim}

\subsection{Some functional spaces and analysis results}
We make frequent use of the divergence and curl operators and 
denote these by $\Div$ and $\Curl$, respectively.
In the 2D case, 
we will denote both the rotation operator taking 
vectors into scalars and the curl operator taking scalars into 
vectors by $\Curl$. 
This confusing but rather standard 
notation greatly simplifies all subsequent arguments allowing identical treatment of
the 2D and 3D cases.  

We will also make use of the spaces
\begin{align*}
	\vc{W}^{\Div, 2}(\Omega) & = 
	\Set{\vc{v} \in \vc{L}^2(\Omega): \Div \vc{v} \in L^2(\Omega)}, \\
	\vc{W}^{\Curl, 2}(\Omega) & = 
	\Set{\vc{v} \in \vc{L}^2(\Omega): \Curl \vc{v} \in \vc{L}^2(\Omega)},
\end{align*}
where $\nu$ denotes the unit outward pointing normal vector on $\partial \Omega$. 
If $\vc{v}\in \vc{W}^{\Div, 2}(\Omega)$ satisfies $\vc{v} \cdot \nu|_{\partial \Omega}=0$, we 
write $\vc{v}\in \vc{W}^{\Div, 2}_{0}(\Omega)$. Similarly, $\vc{v}\in \vc{W}^{\Curl, 2}_0(\Omega)$ 
means $\vc{v}\in\vc{W}^{\Div, 2}(\Omega)$ and $\vc{v} \times \nu|_{\partial \Omega} = 0$. 
In two dimensions, $\vc{w}$ is a scalar function and the 
space $\vc{W}^{\Curl,2}_{0}(\Omega)$ is to be 
understood as $W_{0}^{1,2}(\Omega)$. To define 
weak solutions, we shall use the space
$$
\mcw(\Om) =\Set{\vc{v}\in  \vc{L}^2(\Omega): 
\Div \vc{v} \in L^2(\Omega), 
\Curl \vc{v} \in \vc{L}^2(\Omega), 
\vc{v} \cdot \nu|_{\partial \Omega}=0},
$$
which coincides with $ \vc{W}^{\Div, 2}_{0}(\Omega)\cap \vc{W}^{\Curl, 2}(\Omega)$. 
The space $\mcw(\Om)$ is equipped with the norm $\norm{\vc{v}}_{\mcw}^2=\norm{\vc{v}}_{\vc{L}^2(\Omega)}^2
+\norm{\Div \vc{v}}_{\vc{L}^2(\Omega)}^2+\norm{\Curl \vc{v}}_{\vc{L}^2(\Omega)}^2$. 
It is known that $\norm{\cdot}_{\mcw}$ is equivalent to the $\vc{H}^1$ norm 
on the space $\Set{v\in \vc{H}^1(\Omega): \vc{v} 
\cdot \nu|_{\partial \Omega}=0}$, see, e.g., \cite{Lions:1998ga}.

The space $\mcw(\Omega)$ admits a unique orthogonal Hodge decomposition 
\begin{equation}\label{eq:hodge-cont}
	\mcw(\Omega) = \Curl S(\Omega) + D\Delta^{-1}L^2_{0}(\Omega),
\end{equation}
where $S(\Omega)= \{\vc{v} \in \vc{W}^{1,2}(\Omega): 
\Curl \vc{v} \in \vc{W}^{1,2}(\Omega)\}$, $\Delta^{-1}$ 
is the inverse Neumann Laplace operator, and $L^2_0$ denotes the $L^2$ functions 
on $\Om$ that have zero mean.

For the convenience of the reader we list some basic 
functional analysis results to be used
in the subsequent arguments (for proofs, see, e.g.,\cite{Feireisl:2004oe}). 
Throughout the paper we use overbars to denote weak limits, with the 
underlying spaces being (silently) given by the context.

\begin{lemma}\label{lem:prelim} 
Let $O$ be a bounded open subset of $\R^M$, $M\ge 1$.  
Suppose $g\colon \R\to (-\infty,\infty]$ is a lower semicontinuous 
convex function and $\Set{v_n}_{n\ge 1}$ is a sequence of 
functions on $O$ for which $v_n\weakto v$ in $L^1(O)$, $g(v_n)\in L^1(O)$ for each 
$n$, $g(v_n)\weakto \overline{g(v)}$ in $L^1(O)$. Then 
$g(v)\le \overline{g(v)}$ a.e.~on $O$, $g(v)\in L^1(O)$, and
$\int_O g(v)\ dy \le \liminf_{n\to\infty} \int_O g(v_n) \ dy$. 
If, in addition, $g$ is strictly convex on an open interval
$(a,b)\subset \R$ and $g(v)=\overline{g(v)}$ a.e.~on $O$, 
then, passing to a subsequence if necessary, 
$v_n(y)\to v(y)$ for a.e.~$y\in \Set{y\in O\mid v(y)\in (a,b)}$.
\end{lemma}

Let $X$ be a Banach space and denote by $X^\star$ its dual.  The space
$X^\star$ equipped with the weak-$\star$ topology is denoted by
$X^\star_{\mathrm{weak}}$, while $X$ equipped with the weak topology
is denoted by $X_{\mathrm{weak}}$. By the Banach-Alaoglu theorem, a
bounded ball in $X^\star$ is $\sigma(X^\star,X)$-compact.  If $X$
separable, then the weak-$\star$ topology is metrizable on bounded
sets in $X^\star$, and thus one can consider the metric space
$C\left([0,T];X^\star_{\mathrm{weak}}\right)$ of functions $v:[0,T]\to
X^\star$ that are continuous with respect to the weak topology. We
have $v_n\to v$ in $C\left([0,T];X^\star_{\mathrm{weak}}\right)$ if
$\langle v_n(t),\phi \rangle_{X^\star,X}\to \langle v(t),\phi
\rangle_{X^\star,X}$ uniformly with respect to $t$, for any $\phi\in
X$. The following lemma is a consequence of the Arzel\`a-Ascoli
theorem:

\begin{lemma}\label{lem:timecompactness}
Let $X$ be a separable Banach space, and suppose $v_n\colon [0,T]\to
X^\star$, $n=1,2,\dots$, is a sequence for which 
$\norm{v_n}_{L^\infty([0,T];X^\star)}\le C$, for some constant $C$ independent of $n$. 
Suppose the sequence $[0,T]\ni t\mapsto \langle v_n(t),\Phi \rangle_{X^\star,X}$, $n=1,2,\dots$, 
is equi-continuous for every $\Phi$ that belongs to a dense subset of $X$.  
Then $v_n$ belongs to $C\left([0,T];X^\star_{\mathrm{weak}}\right)$ for every
$n$, and there exists a function $v\in 
C\left([0,T];X^\star_{\mathrm{weak}}\right)$ such that along a 
subsequence as $n\to \infty$ there holds $v_n\to v$ in 
$C\left([0,T];X^\star_{\mathrm{weak}}\right)$.
\end{lemma}

In what follows, we will often obtain a priori estimates for a sequence $\Set{v_n}_{n\ge 1}$ 
that we write as ``$v_n\inb X$'' for some functional space $X$. What this really means is that 
we have a bound on $\norm{v_n}_X$ that is independent of $n$.

\subsection{Topological degree in finite dimensions}
Our numerical method constitutes a nonlinear--implicit discrete problem. 
We will prove the existence of a solution to this problem by 
a topological degree argument \cite{Deimling:1985rt}.

Denote by $d(F,\Om,y)$ the $\mathbb{Z}$--valued (Brouwer) degree of a continuous function 
$F:\bar{\Om} \rightarrow \mathbb{R}^M$ at a point $y\in \mathbb{R}^N\backslash F(\partial S)$ 
relative to an open and bounded set $\Omega \subset \mathbb{R}^M$. 
For notational convenience, let us reformulate the definition of degree so that it 
applies directly in our finite element setting. Indeed, below we define 
$d_{S_{h}}(F,\tilde{S}_{h},q_{h})$ with $F:\tilde{S}_{h} \rightarrow S_{h}$ being 
a continuous finite element mapping, $\tilde{S}_{h}$ being a bounded subset 
of a finite element space $S_{h}$, and 
$q_{h}$ being a function in $S_{h}$.

\begin{definition}\label{def:femtop}
Let $S_{h}$ be a finite element space, $\|\cdot \|$ be a norm on this space, and 
introduce the bounded set
$$
\tilde{S}_{h} = \Set{q_{h} \in S_{h}; \|q_{h}\| \leq C},
$$
where $C>0$ is a constant. Let $\{\sigma_{i}\}_{i=1}^M$ be
a basis such that $\operatorname{span}\{\sigma_{i}\}_{i=1}^M = S_{h}$ and 
define the operator $\Pi_{\mathcal{B}}:S_{h}  \rightarrow \mathbb{R}^M$ by 
$$
\Pi_{\mathcal{B}}q_{h} = (q_{1},q_{2}, \ldots, q_{M}),
\qquad q_{h} = \sum_{i=1}^M q_{i}\sigma_{i}.
$$ 
The degree $d_{S_{h}}(F,\tilde{S}_{h},q_{h})$ of a continuous 
mapping $F:\tilde{S}_{h} \rightarrow S_{h}$ at $q_{h} 
\in S_{h}\backslash F(\partial \tilde{S}_h)$
relative to $\tilde{S}_{h}$ is defined as
\begin{equation*}
	d_{S_{h}}(F,\tilde{S}_{h},q_{h}) 
	= d\left( \Pi_{\mathcal{B}} F(\Pi_{\mathcal{B}}^{-1}), \Pi_{\mathcal{B}}
	\tilde{S}_{h}, \Pi_{\mathcal{B}}q_{h}\right).
\end{equation*}
\end{definition}

The next lemma is a consequence of the properties 
of the degree $d(F,\Omega,y)$, cf.~\cite{Deimling:1985rt}.
\begin{lemma}\label{lem:degree}
Fix a finite element space $S_{h}$, and let  $d_{S_{h}}(F, \tilde{S}_{h},q_{h})$ 
be the associated degree of Definition \ref{def:femtop}. The following properties hold:
\begin{enumerate}
	\item{} $d_{S_{h}}(F,\tilde{S}_{h},q_{h})$ does not depend on the choice of basis for $S_{h}$.
	\item{} $d_{S_{h}}(\mathrm{Id},\tilde{S}_{h}, q_{h}) = 1$.
	\item{} $d_{S_{h}}(H(\cdot,\alpha),\tilde{S}_{h},q_{h}(\alpha))$ is 
	independent of $\alpha \in J:=[0,1]$ for $H\!:\!\tilde{S}_{h}\times J \rightarrow S_{h}$ 
	continuous, $q_h:J \rightarrow S_{h}$ continuous, and 
	$q_{h}(\alpha) \notin H(\partial \tilde{S}_{h},\alpha)$ $\forall \alpha \in [0,1]$.
	\item{}$d_{S_{h}}(F,\tilde{S}_{h},q_{h}) \neq 0 \iff F^{-1}(q_{h}) \neq \emptyset$.
\end{enumerate}
\end{lemma}

\subsection{Weak and renormalized solutions}

\begin{definition}[Weak solutions]\label{def:weak}
We say that a pair $(\vrho,\vc{u})$ of functions constitutes a weak solution
of the semi-stationary compressible Stokes system \eqref{eq:contequation}--\eqref{eq:momentumeq} 
with initial data \eqref{initial} and Navier-slip type boundary conditions 
\eqref{eq:bc-normal}--\eqref{eq:bc-navierslip} 
provided the following conditions hold:

\begin{enumerate}
	\item $(\vrho,\vc{u}) \in L^\infty(0,T;L^\gamma(\Omega))\times L^2(0,T;\mcw(\Om))$;

	\item $\vrho_{t} + \Div (\vrho\vc{u}) = 0$ in the weak 
	sense, i.e, $\forall \phi \in C^\infty([0,T)\times\overline{\Omega})$,
	\begin{equation}\label{eq:weak-rho}
		\int_{0}^T\int_{\Omega}\vrho \left( \phi_{t} + \vc{u}D\phi\right)\ dxdt
		+ \int_{\Omega}\vrho_{0}\phi|_{t=0}\ dx = 0;
	\end{equation}
		
	\item $-\mu \Delta \vc{u} - \lambda D\Div\vc{u} + Dp(\vrho) = \vc{f}$ in 
	the weak sense, i.e, $\forall \vc{\phi} \in \vc{C}^\infty((0,T)\times\overline{\Omega})$ for which 
	$\vc{\phi} \cdot \nu = 0$ on $(0,T)\times \partial \Omega$,
	\begin{equation*}
		\int_{0}^T\int_{\Omega}\mu\Curl \vc{u} \Curl \vc{\phi} 
		+ \left[(\mu + \lambda)\Div \vc{u}-p (\vrho)\right]\Div \vc{\phi}\ dxdt 
		= \int_{0}^T\int_{\Omega}\vc{f}\vc{\phi}\ dxdt.
	\end{equation*}
\end{enumerate}
\end{definition}

For the convergence analysis we shall also need the DiPerna-Lions concept of 
renormalized solutions of the continuity equation.  

\begin{definition}[Renormalized solutions]
\label{renormlizeddef}
Given $\vc{u}\in L^2(0,T;\mcw(\Om))$, we say that $\vrho\in  L^\infty(0,T;L^\gamma(\Omega))$ 
is a renormalized solution of \eqref{eq:contequation} provided 
$$
B(\vrho)_t + \Div \left(B(\vrho)\vc{u}\right) + b(\vrho)\Div \vc{u}=0
\quad \text{in the weak sense on $\cDom$,}
$$
for any $B\in C[0,\infty)\cap C^1(0,\infty)$ with 
$B(0)=0$ and $b(\vrho) := \vrho B'(\vrho) - B(\vrho)$.
\end{definition}

We shall need the following lemma.

\begin{lemma}\label{lemma:feireisl}
Suppose $(\vrho,\vc{u})$ is a weak solution according to Definition \ref{def:weak}.
If $\vrho \in L^2((0,T)\times \Omega))$, then $\vrho$ is a renormalized solution 
according to Definition \ref{renormlizeddef}.
\end{lemma}

\begin{proof}
Let $(\vrho,\vc{u})$ be a weak solution. Then $\vc{u}\in L^2(0,T;\vc{H}^1(\Om))$.
As the boundary of $\Om$ is Lipschitz, the velocity field $\vc{u}(t)$ can be 
extended to the full space $\R^3$ such that $\tilde{\vc{u}}(t)|_{\Om}=\vc{u}(t)$ and 
$$
\norm{\tilde{\vc{u}}}_{L^2(0,T;\vc{H}^1(\R^N))}
\le C(\Om)\norm{\vc{u}}_{L^2(0,T;\vc{H}^1(\Om))},
$$
where $\tilde{\vc{u}}(t)$ denotes the extension of $\vc{u}(t)$. 
If we extend $\vrho(t)$ to $\R^N$ by setting  
$\tilde{\vrho}(t)=\vrho(t)\vc{1}_{\Omega}$, we get
$$
\tilde{\vrho}_{t} + \Div (\tilde{\vrho}\,\tilde{\vc{u}}) = 0
\quad \text{in the weak sense on $[0,T)\times\R^3$.}
$$
Now, to conclude the proof, we appeal to a well-known 
lemma from \cite{Lions:1998ga} stating that the square-integrable 
weak solution $\tilde{\vrho}$ is also a renormalized solution.
\end{proof}

\subsection{A mixed formulation}

In view of the Navier-slip boundary condition \eqref{eq:bc-navierslip}, it 
is natural to introduce the vorticity $\vc{w} = \Curl \vc{u}$ as an independent 
variable, thereby turning the velocity equation into \eqref{eq:vorticity-form}. 
This immediately leads to the following mixed formulation, which acts as a 
motivation for our choice of numerical method: Determine functions 
$$
(\vc{w},\vc{u}) \in L^2(0,T;\vc{W}_{0}^{\Curl, 2}(\Omega))
\times L^2(0,T;\vc{W}^{\Div, 2}_{0}(\Omega))
$$
such that
\begin{equation}\label{def:mixed-weak}
	\begin{split}
		& \int_{0}^T\int_{\Omega}\mu\Curl \vc{w} \vc{v}
		+ \left[(\mu + \lambda)\Div \vc{u} - p(\vrho)\right]\Div \vc{v}\ dxdt
		= \int_{0}^T\int_{\Omega}\vc{f}\vc{v}\ dxdt, \\
		& \int_{0}^T\int_{\Omega} \vc{w}\vc{\eta}-\Curl \vc{\eta} \vc{u}\ dxdt = 0, 
	\end{split}
\end{equation}
for all $(\eta,\vc{v}) \in  L^2(0,T;\vc{W}_{0}^{\Curl, 2}(\Omega))
\times L^2(0,T;\vc{W}^{\Div, 2}_{0}(\Omega))$. 

In order to arrive at the weak formulation \eqref{def:mixed-weak}, we have utilized 
the integration by parts formula
\begin{equation}\label{eq:integrationbyparts}
	\int_\Om \vc{\eta}\Curl \vc{u}~dx = \int_\Om \vc{u} \Curl \vc{\eta}~dx + \int_{\partial \Om} \vc{u} (\vc{\eta} \times \nu)~dS(x).
\end{equation}
It follows as an immediate consequence of the Stokes Theorem and will be applied 
multiple times throughout the paper.

The upcoming goal is to prove that a sequence of approximate solutions, denoted by 
$\Set{(\vrho_h,\vc{w}_h,\vc{u}_h)}_{h>0}$, converge 
to a limit $(\vrho,\vc{w},\vc{u})$ satisfying 
\eqref{eq:weak-rho} and \eqref{def:mixed-weak}; the term ``converge" is 
made precise in a forthcoming section. Having constructed such
a limit, it follows immediately that the pair $(\vrho,\vc{u})$ is a weak solution 
according to Definition \ref{def:weak}, thereby completing the analysis.

\subsection{Finite element spaces and some basic results}
Upon inspection of the spatial spaces entering the weak formulations stated 
above, we see that they can be related through a De Rham sequence. 
In two dimensions this reads
\begin{equation*}
{\small 
	\begin{CD}
		0 @> \subset >> W^{1,2}_{0} @> \Curl\ >> \vc{W}^{\Div, 2}_{0}@> 
		\Div\ >> L^2_{0} @> >> 0,
	\end{CD}
}
\end{equation*}
while in three dimensions the corresponding sequence is
\begin{equation*}
{\small 
\begin{CD}
0 @> \subset >> W^{1,2}_{0} @> \operatorname{grad} >> 
\vc{W}^{\Curl,2}_{0} @> \Curl\ >> \vc{W}^{\Div, 2}_{0} @> \Div\ >> L^2_{0} @> >> 0.
\end{CD}
}
\end{equation*}
These sequences are exact in the sense that the null 
space of one operator exactly matches the image of the next.
This perspective on the spaces is actually useful as we see that this is precisely how the 
quantities $\vc{w}, \vc{u}$, and $\vrho$ relate to each other. 

It follows from \eqref{def:mixed-weak} that the vorticity $\vc{w}$ is 
decoupled from the density $\vrho$, which is an important consequence 
of our choice boundary condition and this fact is 
of relevance to the convergence analysis. 
Moreover, the subsequent analysis 
relies heavily on the solvability of 
the problem (or more precisely a discrete version of it)
$$
\Div \vc{v} = q, \quad \vc{v}\cdot \nu = 0 \quad \textrm{on $\pOm$},
$$
for some given right-hand side $q$ in $L^2$. In particular, it 
is important for us to extract from this problem some control on $\Curl \vc{v}$.  
From the above De Rham sequence, we see 
immediately that there exists solution $\vc{v}$ which is weakly curl free. In the continuous 
setting this is enough to conclude that $\Curl \vc{v} = 0$; indeed, 
the Hodge decomposition \eqref{eq:hodge-cont} combined with the fact that $\vc{v}$ is weakly
curl free implies $v = Ds$ for some scalar $s$. 

Motivated by these remarks, we shall in the next section present a 
numerical method that utilizes finite element spaces satisfying a 
discrete version of the above De Rham sequence. More precisely, we will replace 
$\vc{W}^{\Curl, 2}_0$ and $\vc{W}^{\Div, 2}_{0}$ by the lowest order
N{\'e}d{\'e}lec finite element spaces of the first kind (but other spaces are possible)
with vanishing degrees of freedom at the boundary $\partial \Om$. Let us 
denote these spaces by $\vc{W}_{h}$ and $\vc{V}_{h}$ respectively. 
It is well known that the spaces $\vc{W}_{h},\vc{V}_{h}$ together 
with the space $Q_{h}$ of piecewise constants (cf.~the ensuing section 
for missing details) satisfies in three dimensions the following 
exact discrete De Rham sequence:
\begin{equation*}
	{\small 
	\begin{CD}
		0 @> \subset >> S_{h} @> 
		\operatorname{grad} >> \vc{W}_h @> \Curl\ >> \vc{V}_h @> 
		\Div\ >> Q_{h}\cap L^2_0(\Om) @> >> 0,
	\end{CD}}
\end{equation*}
where $S_{h}$ is the usual scalar linear Lagrange element. In the two 
dimensional case this sequence still holds, but now the 
spaces $S_{h}$ and $\vc{W}_{h}$ are equal and thus the sequence does
not contain the gradient operator.  All finite element spaces are defined with 
respect to a given tetrahedral mesh $E_{h}$ of $\Omega$. 

We introduce the canonical interpolation operators: 
\begin{equation*}
	\begin{split}
	&\Pi_{h}^S:W^{1,2}_{0}\cap ~W^{2,2} \rightarrow S_{h}, \quad 
	\Pi_{h}^W:\vc{W}^{\Curl,2}_{0}\cap ~\vc{W}^{2,2} \rightarrow \vc{W}_{h}, \\
	&\Pi_{h}^V: \vc{W}_{0}^{\Div, p}\cap \vc{W}^{1,2} \rightarrow \vc{V}_{h}, \quad 
	\Pi_{h}^Q: L^2_{0} \rightarrow Q_{h},	
\end{split}
\end{equation*}
using the available degrees of freedom of the involved spaces. That is,
the operators are defined
\begin{equation*}
	\begin{split}
		\left(\Pi_h^S s\right)(x_i) &= s(x_i), \quad \forall x_i \in \mathcal{N}_h; \\
		\int_e \left(\Pi_h^W \vc{w}\right)\times \nu~dS(x) 
			&= \int_e  \vc{w}\times \nu~dS(x), 
				\quad \forall e \in \mathcal{E}_h;\\
		\int_\Gamma \left(\Pi_h^V \vc{v}\right)\cdot \nu~dS(x) &= \int_\Gamma \vc{v} \cdot \nu~dS(x), 
				\quad \forall \Gamma \in \Gamma_h;\\
		\int_E \Pi_h^Q q~dx &= \int_E q ~dx, \quad \forall E \in E_h,
	\end{split}
\end{equation*}
where $\Gamma_h$, $\mathcal{E}_h$, and $\mathcal{N}_h$, denote the set of faces, edges, and vertices, respectively, of $E_h$.
Then it is well known that the following diagram commutes:
\begin{equation*}
	{\small 
	\begin{CD}
		 W^{1,2}_{0}\cap W^{2,2} @> 
		\operatorname{grad} >> \vc{W}^{\Curl,2}_{0}\cap ~\vc{W}^{2,2} @> \Curl\ >> 
		\vc{W}^{\Div, p}_{0}\cap \vc{W}^{1,2} @> \Div\ >> \vc{L}^2_{0}\\
		 @V \Pi_{h}^SVV    @V\Pi_{h}^WVV   @V\Pi_{h}^{V}VV          
		@V\Pi_{h}^QVV  \\ 
		S_{h} @>\operatorname{grad} >> \vc{W}_{h} @> \Rot\ >> 
		\vc{V}_{h} @> \Div\ >> Q_{h}.
	\end{CD}
	}
\end{equation*}

\begin{remark}
	The interpolation operators $\Pi_h^S$, $\Pi_h^W$, and $\Pi_h^V$, are 
	defined on function spaces with enough regularity to ensure that the 
	corresponding degrees of freedom are functionals on these spaces. 
	This is reflected in writing $\vc{W}^{\Curl,2}_{0}\cap \vc{W}^{2,2}$
	instead of merely $\vc{W}^{\Curl,2}$ and so on.
\end{remark}
In view of the above commuting diagram, we can define the spaces orthogonal to the range of the previous operator, i.e., 
\begin{align*}
	\vc{W}_{h}^{0,\perp} & := \{\vc{w}_{h} \in \vc{W}_{h}; \Curl \vc{w}_{h} =0 \}^\perp\cap \vc{W}_{h}, \\
	\vc{V}_{h}^{0,\perp} & := \{\vc{v}_{h} \in \vc{V}_{h}; \Div \vc{v}_{h} = 0\}^\perp \cap \vc{V}_{h},
\end{align*}
to obtain decompositions
\begin{align}
	\vc{W}_{h} &= DS_{h} +\vc{W}_{h}^{0, \perp}, \nonumber 
	\\
	\vc{V}_{h}  & = \Curl \vc{W}_{h} +\vc{V}_{h}^{0, \perp}, \label{eq:Vh-decomp}
\end{align}
and the discrete Poincar\'e inequalities
\begin{align}
	\label{Poincare1}
	\norm{\vc{v}_{h}}_{\vc{L}^2(\Omega)} & \leq 
	C\norm{\Div \vc{v}_{h}}_{L^2(\Omega)},
	\quad \forall \vc{v} \in \vc{V}_{h}^{0,\perp},\\
	\label{Poincare2}
	\norm{\vc{w}_{h}}_{\vc{L}^2(\Omega)} 
	& \leq C\norm{\Curl \vc{w}_{h}}_{L^2(\Omega)},
	\quad \forall \vc{w} \in \vc{W}_{h}^{0,\perp}.
\end{align}	
Thus, with this configuration of elements we are able to perform unique 
Hodge type decompositions of the discrete vector fields. 
As an example, we immediately have
the existence of a 
function $\vc{v}_{h} \in \vc{V}^{0,\perp}_{h}$ satisfying
\begin{equation*}
	\Div \vc{v}_{h}|_E = q_{h}|_E, ~\forall E \in E_h,
\end{equation*}
for any given $q_h \in Q_h\cap \Set{\int_\Om q_h~dx = 0}$.

The following lemma summarizes well--known error estimates
satisfied by the interpolation operators. 
The estimates are derived from the
Bramble--Hilbert lemma using scaling arguments. 
We however note that care must be taken when mapping functions in $\vc{W}_h$ 
and $\vc{V}_h$ to a reference element (cf. \cite{Brezzi:1991lr,Nedelec:1980ec}).

\begin{lemma}\label{lemma:interpolation}
There exists a constant $C>0$, depending only on the shape 
regularity of $E_h$ and the size of $\Om$, such that for any $1\leq p<\infty$,
\begin{align*}
	& \norm{\phi - \Pi_h^Q\phi}_{L^p(\Om)} \leq C h\norm{D\phi}_{\vc{L}^p(\Om)}, \\
	& \norm{\vc{v} - \Pi_h^V \vc{v}}_{\vc{L}^p(\Om)}
	+h\norm{\Div (\vc{v} - \Pi_h^V \vc{v})}_{L^p(\Om)}
	\leq C h^{s}\norm{D^{s}\vc{v}}_{\vc{L}^p(\Om)}, \quad r=1,2, \\
	& \norm{\vc{w} - \Pi_h^W \vc{w}}_{\vc{L}^p(\Om)}
	+ h\norm{\Curl (\vc{w} - \Pi_h^W \vc{w})}_{\vc{L}^p(\Om)} 
	\leq C h^{s}\|D^{s}\vc{w}\|_{\vc{L}^p(\Om)}, \quad s=1,2,
\end{align*}
for all $\phi \in W^{1,p}(\Om),\vc{v}\in W^{s,p}(\Om)$, 
and $\vc{w} \in W^{2,p}(\Om)$.
\end{lemma}

In what follows, we will need the following lemma. 
It follows from scaling arguments and the equivalence of 
finite dimensional norms.
\begin{lemma}\label{lemma:inverse}
There exists a constant $C>0$, depending only on the shape 
regularity of $E_h$, such that for $1\leq q,p \leq \infty$,
and $r= 0,1$,
\begin{equation*}
	\norm{\phi_h}_{W^{r,p}(E)} 
	\leq C h^{-r + \min\{0, \frac{N}{p}-\frac{N}{q}\}}
	\norm{\phi_h}_{L^q(E)}, 
\end{equation*}
for any $E \in E_h$ and all polynomial functions $\phi_h \in \mathbb{P}_k(E)$, $k=0,1,\ldots$.
\end{lemma}

The next result follows from scaling arguments and the trace theorem.
\begin{lemma}\label{lemma:edgebounds}
Fix any $E \in E_{h}$ and let $\phi \in W^{1,2}(E)$ be arbitrary. 
There exists a constant $C>0$, depending only on the shape regularity of $E_h$ such that, 
$$
\|\phi\|_{L^2(\Gamma)} \leq Ch^{-\frac{1}{2}}\left(\|\phi\|_{L^2(E)} + h\|D \phi\|_{\vc{L}^2(E)}\right), \quad \forall \Gamma \in \Gamma_{h}\cap \partial E.
$$
\end{lemma}


\section{Numerical method and main result}\label{sec:numerical-method}
In this section we define the numerical method and the state the convergence theorem. The proof 
of this theorem is deferred to subsequent sections.

Given a time step $\Dt>0$, we discretize the time interval $[0,T]$ in 
terms of the points $t^m=m\Dt$, $m=0,\dots,M$, where we assume that $M\Dt=T$.  
Regarding the spatial discretization, we let $\{E_{h}\}_{h}$ be a shape regular 
family of tetrahedral meshes of $\Omega$,  where $h$ is the maximal diameter.
It will be a standing assumption that $h$ and $\Delta t$ are 
related such that $\Delta t = c h$, for some constant $c$. 
By shape regular we mean that there exists a constant 
$\kappa > 0$ such that every $E \in E_{h}$ contains a ball of radius 
$\lambda_{E} \geq \frac{h_{E}}{\kappa}$, where $h_{E}$ is the diameter of $E$. 
Furthermore, we let $\Gamma_{h}$ denote the set of faces in $E_{h}$. 
Throughout the paper, we will use the three dimensional terminology 
(tetrahedron, face, etc.) to denote both the three dimensional 
and the two dimensional case (triangle, edge, etc). 

On each element $E \in E_{h}$, we denote by $Q(E)$ the constants on $E$. 
The functions that are piecewise constant with respect to the elements of a mesh 
$E_{h}$ are denoted by $Q_h=Q_h(\Om)$. Next, on each $E \in E_{h}$, we 
denote by $\vc{W}(E)$ the lowest order space of \emph{curl--conforming N{\'e}d{\'e}lec} 
polynomials of first kind \cite{Nedelec:1980ec}. 
In two dimensions, $\vc{W}(E)$ is the space of linear scalar polynomials on $E$ and
is totally determined by it's value at the vertices of $E$. 
In three dimensions, each member of $\vc{W}(E)$ is of the form
\begin{equation*}
	\vc{a} + \vc{b} \times \begin{pmatrix}x \\ y \\ z \end{pmatrix}, 
	\quad \vc{a}, \vc{b} \in \mathbb{R}^3,
\end{equation*}
and is totally determined by the following degrees of freedom:
$
	\int_e \vc{w} \cdot \tau_e~dS(x)
$
for all edges (not faces) $e$ of the element $E$, where $\tau_e$ is
the unit tangential vector on $e$.

On each element $E \in E_{h}$, we denote by $\vc{V}(E)$ the 
lowest order space of \emph{div--conforming N{\'e}d{\'e}lec} 
polynomials of first kind \cite{Nedelec:1980ec}. In two dimensions, it is 
the Raviart--Thomas polynomial space on $E$. 
Each member of $\vc{V}(E)$ is of the form
\begin{equation*}
	\vc{a} + b\begin{pmatrix}x \\ y \\ z \end{pmatrix}, 
	\quad \vc{a} \in \mathbb{R}^3, b \in \mathbb{R},
\end{equation*}
and is totally determined by the following degrees of freedom: 
$
\int_\Gamma \vc{v} \cdot \nu~dS(x), 
$
for all faces $\Gamma$ of the element $E$, where $\nu$ is a unit
normal vector on $\Gamma$.

The element spaces $\vc{W}_{h}=\vc{W}_{h}(\Omega)$ and $\vc{V}_{h}=\vc{V}_{h}(\Omega)$ 
are formed on the entire mesh $E_{h}$ by matching the 
degrees of freedom of the polynomial space $\vc{W}(E)$ and $\vc{V}(E)$, respectively,  
on each face $\Gamma \in \Gamma_{h}$. 
In addition, we incorporate the boundary conditions by letting the degrees of freedom of the 
spaces $\vc{W}_{h}$ and $\vc{V}_{h}$ vanish at the faces on the boundary.

Before defining our numerical method, we shall need to introduce some 
additional notation related to the discontinuous Galerkin scheme. 
Concerning the boundary $\partial E$ of an element $E$, we write $f_{+}$ 
for the trace of the function $f$ achieved from within the element $E$ 
and $f_{-}$ for the trace of $f$ achieved from outside $E$. 
Concerning a face $\Gamma$ that
is shared between two elements $E_{-}$ and $E_{+}$, we will write $f_{+}$ for 
the trace of $f$ achieved from within $E_{+}$ and $f_{-}$ for the trace
of $f$ achieved from within $E_{-}$. Here $E_{-}$ and $E_{+}$ are 
defined such that $\nu$ points from $E_{-}$ to $E_{+}$, where $\nu$ is 
fixed (throughout) as one of the two possible 
normal components on each face $\Gamma$.
We also write $[f]_{\Gamma}= f_{+} - f_{-}$ for the jump of $f$ 
across the face $\Gamma$, while forward time-differencing 
of $f$ is denoted by $[f^m] = f^{m+1} - f^m$. 
To denote the set of inner faces of $\Gamma_h$ we will use the 
notation
$\Gamma_h^I = \Set{\Gamma \in \Gamma_h; \Gamma \not \subset \partial \Om}$.

Let us now define our numerical method for the semi-stationary 
Stokes system \eqref{eq:contequation}--\eqref{eq:momentumeq} 
augmented with the boundary conditions \eqref{eq:bc-normal} and \eqref{eq:bc-navierslip} (note, however, 
that in the definition below the boundary conditions are built into the 
finite element spaces and not listed explicitly).

\begin{definition}[Numerical method]\label{def:num-scheme}
Let $\Set{\vrho^0_h(x)}_{h>0}$ be a sequence in $Q_{h}(\Omega)$ 
that satisfies $\vrho_h^0>0$ for each fixed $h>0$ 
and $\vrho^0_h\to \vrho_0$ a.e.~in $\Om$ and in $L^1(\Om)$ 
as $h\to 0$. Set $\vc{f}_{h}(t,\cdot)=\vc{f}_{h}^m(\cdot) 
:= \frac{1}{\Delta t}\int_{t^{m-1}}^{t^m} \Pi_h^Q \vc{f}(s,\cdot) \ ds$, 
for $t\in (t_{m-1},t_m)$, $m=1,\ldots,M$.

Now, determine functions 
$$
(\vrho^m_{h},\vc{w}^m_{h},\vc{u}^m_{h}) \in Q_{h}(\Omega)\times\vc{W}_{h}(\Omega)
\times \vc{V}_{h}(\Omega), \quad m=1,\dots,M,
$$
such that for all $\phi_{h} \in Q_{h}(\Omega)$,
\begin{equation}\label{FEM:contequation}
	\begin{split}
		&\int_\Omega \vrho^m_h \phi_{h}\ dx
		-\Delta t\sum_{\Gamma \in \Gamma^I_h}\int_\Gamma \left(\vrho^m_{-}(\vc{u}^{m}_h \cdot \nu)^+
		+\vrho^m_+(\vc{u}^{m}_h \cdot \nu)^-\right)[\phi_{h}]_\Gamma\ dS(x)
		\\ & \qquad = \int_\Omega \vrho^{m-1}_h\phi_{h}\ dx, 
	\end{split}
\end{equation}
and for all $(\vc{\eta}_{h},\vc{v}_{h}) \in \vc{W}_{h}(\Omega)\times \vc{V}_{h}(\Omega)$,
\begin{equation}\label{FEM:momentumeq}
	\begin{split}
		&\int_{\Omega}\mu\Curl \vc{w}^m_{h}\vc{v}_{h} + \left[(\mu + \lambda)\Div \vc{u}^m_{h}
		-p(\vrho^m_{h})\right]\Div \vc{v}_{h}\ dx
		= \int_{\Omega}\vc{f}^m_{h}\vc{v}_{h}\ dx, \\
		&\int_{\Omega}\vc{w}^m_{h}\vc{\eta}_{h} -  \vc{u}^m_{h}\Curl \vc{\eta}_{h} \ dx =0,
	\end{split}
\end{equation}
for $m=1,\dots,M$. 

In \eqref{FEM:contequation}, $(\vc{u}_{h} \cdot \nu)^+=\max\{\vc{u}_{h} \cdot \nu, 0\}$ 
and $(\vc{u}_{h} \cdot \nu)^+ = \min\{\vc{u}_{h} \cdot \nu, 0\}$, so that 
$\vc{u}_{h} \cdot \nu=(\vc{u}_{h} \cdot \nu)^++(\vc{u}_{h} \cdot \nu)^-$, i.e., in the 
evaluation of $\vrho(\vc{u} \cdot \nu)$ at the face $\Gamma$ the 
trace of $\vrho$ is taken in the upwind direction. 
\end{definition}

\begin{remark}\label{rem:E-VS-Gamma}
Using the identity
\begin{align*}
	&\Delta t\sum_{E \in E_{h}}\int_{\binner}\left(\vrho^m_{+}(\vc{u}_{h}^{m} \cdot \nu)^+
	+\vrho^m_{-}(\vc{u}_{h}^{m} \cdot \nu)^-\right)\phi_{h} \ dS(x) \\
	&\qquad 
	= -\Delta t\sum_{\Gamma \in \Gamma^I_{h}}\int_{\Gamma}\left(\vrho^m_{+}(\vc{u}_{h}^{m} \cdot \nu)^+
	+\vrho^m_{-}(\vc{u}_{h}^{m} \cdot \nu)^-\right)[\phi_{h}]_\Gamma \ dS(x).
\end{align*}
we can state \eqref{FEM:contequation} on the following form:
\begin{equation}\label{FEM:contequation-newform}
	\begin{split}
		&\int_\Omega \vrho^m_h \phi_{h}\ dx
		+\Delta t\sum_{E \in E_{h}}\int_{\binner}\left(\vrho^m_{+}(\vc{u}_{h}^{m} \cdot \nu)^+
		+\vrho^m_{-}(\vc{u}_{h}^{m} \cdot \nu)^-\right)\phi_{h} \ dS(x)
		\\ & \qquad = \int_\Omega \vrho^{m-1}_h\phi_{h}\ dx.
	\end{split}
\end{equation}
\end{remark}

For each fixed $h>0$, the numerical solution 
$\Set{(\vrho^m_{h},\vc{w}^m_{h},\vc{u}^m_{h})}_{m=0}^M$
is extended to the whole of $(0,T]\times \Omega$ by setting
\begin{equation}\label{eq:num-scheme-II}
	(\vrho_{h},\vc{w}_{h},\vc{u}_{h})(t)=(\vrho^m_{h},\vc{w}^m_{h},\vc{u}^m_{h}), 
	\qquad t\in (t_{m-1},t_m], \quad m=1,\dots,M.
\end{equation}
In addition, we set $\vrho_{h}(0)= \vrho^0_{h}$. 

Our main result is that, passing if necessary to a subsequence, 
$\Set{(\vrho_{h},\vc{w}_{h},\vc{u}_{h})}_{h>0}$ 
converges to a weak solution. More precisely, there holds

\begin{theorem}[Convergence]\label{theorem:mainconvergence}
Suppose $\vc{f}\in \vc{L}^2(\Dom)$ and $\vrho_0\in L^\gamma(\Om)$,  $\gamma > 1$.
Let $\Set{(\vrho_{h},\vc{w}_{h},\vc{u}_{h})}_{h>0}$ be a sequence of numerical solutions 
constructed according to \eqref{eq:num-scheme-II} and Definition \ref {def:num-scheme}. 
Then, passing if necessary to a subsequence as $h\to 0$, $\vc{w}_{h} \weak \vc{w}$ 
in $L^2(0,T;\vc{W}^{\Curl,2}_{0}(\Om))$, $\vc{u}_{h} \weak \vc{u}$ 
in $L^2(0,T;\vc{W}^{\Div, 2}_{0}(\Omega))$, $\vrho_{h}\vc{u}_{h} \weak \vrho\vc{u}$ 
in the sense of distributions on $\Dom$, and $\vrho_{h} \rightarrow \vrho$
a.e.~in $\Dom$, where the limit $(\vrho,\vc{w}, \vc{u})$ satisfies the mixed formulation
\eqref{def:mixed-weak}, and consequently $(\vrho,\vc{u})$ is also a weak 
solution according to Definition \ref{def:weak}.
\end{theorem}

This theorem will be an immediate consequence 
of the results stated and proved 
in Sections \ref{sec:existence-method}--\ref{sec:conv}.

\section{Numerical method is well defined}\label{sec:existence-method}
In this section we show that there exists a solution to the discrete problem 
given in Definition \ref {def:num-scheme}.  However, we commence by obtaining a 
positive lower bound for the density, recalling that the 
approximate initial density $\vrho^0_h(\cdot)$ is strictly positive.

\begin{lemma} \label{lemma:vrho-props}
Fix any $m=1,\dots,M$, and suppose $\vrho^{m-1}_{h} 
\in Q_{h}(\Om)$, $\vc{u}^m_{h} \in \vc{V}_{h}(\Om)$ 
are given bounded functions. Then the solution $\vrho^{m}_{h} \in Q_{h}(\Om)$ of 
the discontinuous Galerkin scheme \eqref{FEM:contequation} satisfies
$$
\min_{x \in \Omega}\vrho_{h}^m(x) \geq \min_{x \in \Omega}\vrho_{h}^{m-1}(x)
\left(\frac{1}{1 + \Delta t \|\Div \vc{u}^m_{h}\|_{L^\infty(\Omega)}}\right).
$$
Consequently, if $\vrho^{m-1}_{h}(\cdot)>0$, 
then $\vrho^{m}_{h}(\cdot)>0$.
\end{lemma}

\begin{proof}
Let $\tilde{E} \in E_{h}$ be such that $\vrho_{h}^m\big|_{\tilde{E}} \leq 
\vrho_{h}^m\big|_{E}$ $\forall E \in E_{h}$, and insert 
into \eqref{FEM:contequation-newform} the test 
function $\phi_{h} \in Q_{h}(\Om)$, defined by
$$
\phi_{h}(x) = 
\begin{cases}
	\frac{1}{|\tilde{E}|}, & x \in \tilde{E}, \\
	0, & \textrm{otherwise}.
\end{cases}
$$
Integrating by parts then yields
\begin{align*}
	\vrho_h^m\big |_{\tilde{E}} &= 
	-\frac{\Delta t}{|\tilde{E}|}\int_{\partial \tilde{E}\setminus \partial \Om}\left(\vrho^m_{+}(\vc{u}^m_{h} \cdot \nu)^+
	+ \vrho_{-}^m(\vc{u}_{h} \cdot \nu)^-\right) dS(x) + \vrho_{h}^{m-1}\big|_{\tilde{E}} \\
	&= -\Delta t \left(\vrho^m_{h} \Div \vc{u}_{h}^m\right)\big|_{\tilde{E}}
	- \frac{\Delta t}{|\tilde{E}|}\int_{\partial \tilde{E}\setminus \partial \Om}
	(\vrho_{-}^m - \vrho_{+}^m)(\vc{u}_{h} \cdot \nu)^- dS(x)
	+ \vrho_{h}^{m-1}\big|_{\tilde{E}},
	\\ & \ge  -\Delta t \left(\vrho^m_{h} \Div \vc{u}_{h}^m\right) \big|_{\tilde{E}}
	+ \vrho_{h}^{m-1}\big|_{\tilde{E}}, 
\end{align*}
where we have also used the relation $\vc{u}_{h} \cdot \nu=(\vc{u}_{h} \cdot \nu)^++(\vc{u}_{h} \cdot \nu)^-$, 
$\vrho^m_{h}$ is constant on $\tilde{E}$, and that $\vrho^m_{h}$ 
attains its minimal value on $\tilde{E}$. Consequently,
\begin{equation*}
	\vrho_{h}^m|_{\tilde{E}} \geq 
	\vrho_{h}^{m-1}\big|_{\tilde{E}}
	\left(\frac{1}{1 + \Delta t \|\Div \vc{u}^m_{h}\|_{L^\infty(\Omega)}}\right).
\end{equation*}
\end{proof}

We now turn to the existence of solutions to our nonlinear--implicit discrete 
problem. We will apply a topological degree argument, thereby reducing the proof to 
exhibiting a solution to a linear problem.

\begin{lemma}
For each fixed $h > 0$, there exists a solution 
$$
(\vrho^m_{h},\vc{w}^m_{h},\vc{u}^m_{h}) 
\in Q_{h}(\Omega)\times\vc{W}_{h}(\Omega)
\times \vc{V}_{h}(\Omega), \quad 
\vrho^m_h(\cdot)>0, \quad m=1,\dots,M,
$$
to the nonlinear--implicit discrete problem 
posed in Definition \ref {def:num-scheme}. 
\end{lemma}

\begin{proof}
We argue by induction. Assume for $m=1,\ldots, k-1$ that 
there exists a solution 
$$
(\vrho^m_{h},\vc{w}^m_{h},\vc{u}^m_{h}) 
\in S_h:=Q^+_{h}(\Omega)\times\vc{W}_{h}(\Omega)
\times \vc{V}_{h}(\Omega)
$$
to the discrete problem of Definition \ref {def:num-scheme}. Here and below we denote by 
$Q_{h}^+(\Omega)$ the strictly positive functions in $Q_{h}(\Omega)$. 
Moreover, the norm $\norm{\cdot}$ on $S_h$ is defined by
$\norm{(\vrho_{h},\vc{w}_{h},\vc{u}_{h})}^2=
\norm{\vrho_h}_{L^2(\Om)}^2+\norm{\vc{w}_{h}}_{L^2(\Om)}^2
+\norm{\vc{u}_{h}}_{L^2(\Om)}^2$.

The claim is that we can find a solution for $m=k$:
\begin{equation}\label{eq:k-solution}
	(\vrho^k_{h},\vc{w}^k_{h},\vc{u}^k_{h}) \in S_h.
\end{equation}
To this end, we introduce the mapping 
\begin{align*}
	&H:Q^+_{h}(\Omega)\times\vc{W}_{h}(\Omega) \times \vc{V}_{h}(\Omega)\times [0,1] 
	\rightarrow Q_{h}(\Omega)\times\vc{W}_{h}(\Omega) \times \vc{V}_{h}(\Omega),
	\\ & H(\vrho_{h},\vc{w}_{h},\vc{u}_{h},\alpha) = 
	\left(z_{h}(\alpha), \vc{y}_{h}(\alpha), \vc{x}_{h}(\alpha)\right),
\end{align*}
where the triplet $(z_{h}(\alpha),\vc{y}_{h}(\alpha), \vc{x}_{h}(\alpha))$
is defined by
\begin{align*}
	\int_{\Omega}z_{h}(\alpha) \phi_{h}\ dx &
	= \alpha\sum_{E \in E_{h}}\int_{\binner}\left(\vrho_+ (\vc{u}_{h}\cdot \nu)^+
	+\vrho_{-}(\vc{u}_{h} \cdot \nu)^-\right)\phi_{h}\ dS(x) \\ 
	&\qquad\qquad + \int_{\Omega}\frac{\vrho_{h} - \vrho_{h}^{k-1}}{\Delta t}\phi_{h}\ dx, 
	\quad   \forall \phi_{h} \in Q_{h}(\Omega),
\end{align*}
\begin{align*}
	\int_{\Omega}\vc{x}_{h}(\alpha) \vc{v}_{h}\ dx &
	= \int_{\Omega} \mu \Curl \vc{w}_{h} \vc{v}_{h} 
	+ (\lambda + \mu)\Div \vc{u}_{h}\Div \vc{v}_{h}\ dx \\ & 
	\qquad - \alpha \int_{\Omega} P(\vrho_{h})\Div \vc{v}_{h}\ dx
	-\int_{\Omega} \vc{f}_{h}\vc{v}_{h}\ dx, \quad \forall \vc{v}_{h} \in \vc{V}_{h}(\Omega),
\end{align*}
\begin{equation*}
	\int_{\Omega}\vc{y}_{h}\vc{\eta}_{h}\ dx 
	= \int_{\Omega}\vc{w}_{h}\vc{\eta}_{h} - \vc{u}_{h}\Curl \vc{\eta}_{h}\ dx, 
	\quad \forall \vc{\eta}_{h}\in \vc{W}_{h}(\Omega).
\end{equation*}

Solving $H(\vrho_{h}, \vc{u}_{h},\vc{w}_{h},1)=0$ is 
equivalent to finding a solution \eqref{eq:k-solution}  
to the nonlinear--implicit discrete problem posed in Definition \ref {def:num-scheme}. 

Let us fix an arbitrary $\alpha \in [0,1]$, and consider a solution 
$(\vrho_{h}(\alpha), \vc{u}_{h}(\alpha),\vc{w}_{h}(\alpha))$ belonging to 
$Q^+_{h}(\Omega)\times\vc{W}_{h}(\Omega)\times \vc{V}_{h}(\Omega)$ of 
the corresponding problem
$$
H(\vrho_{h}(\alpha), \vc{u}_{h}(\alpha), \vc{w}_{h}(\alpha), \alpha) = 0.
$$ 
We claim that there is a constant $C_{\dagger}>0$, independent of $\alpha$, such that
\begin{equation}\label{eq:alphabound}
	\norm{\vc{w}_{h}(\alpha)}_{\vc{W}^{\Curl ,2}(\Omega)}^2 
	+ \norm{\vc{u}_{h}(\alpha)}^2_{\vc{W}^{\Div,2}(\Omega)} 
	+ \norm{\vrho_{h}(\alpha)}^\gamma_{L^\gamma(\Omega)}
	 \leq C_{\dagger}.
\end{equation}
Indeed, repeating the arguments leading to estimate \eqref{eq:lastinstab} 
in Section \ref{sec:basic-est} we conclude that \eqref{eq:alphabound} holds
with $C_{\dagger} = C \|\vc{f}\|_{L^2(0,T;\vc{L}^2(\Omega))}^2 
+ \|\vrho_{h}^{k-1}\|_{L^\gamma(\Omega)}^\gamma$, where the constant 
$C$ is the constant appearing in \eqref{eq:lastinstab}. 
Here, we have also used that $\Delta t = O(h)$. 
Let $\tilde{S}_{h} \subset S_{h}=Q_{h}^+(\Omega)
\times \vc{V}_{h}(\Omega) \times \vc{W}_{h}(\Omega)$ be a ball 
of sufficiently large radius, cf~\eqref{eq:alphabound}.
Then, since every solution of $H(\vrho_{h}, \vc{u}_{h}, \vc{w}_{h}, \alpha)= 0$ 
lies strictly inside $\tilde{S}_{h}$, 
\begin{equation}\label{eq:notzerobcd}
	0 \notin H(\partial \tilde{S}_{h},\alpha), 
	\qquad \forall \alpha \in [0,1].
\end{equation}

We claim that $H(\cdot, \cdot)$ is continuous on $\tilde{S}_{h} \times [0,1]$.
Let $(\vrho_{h}, \vc{u}_{h}, \vc{w}_{h})\in \tilde{S}_{h}$. By equivalence of norms on 
finite dimensional spaces, the functions $\vrho_{h}, \vc{u}_{h}, \vc{w}_{h}$ 
are bounded on $\Om$.  In view of this and 
Lemma \ref{lemma:inverse}, the claim follows. 

By the virtue of \eqref{eq:notzerobcd} and the 
continuity of $H(\cdot, \cdot)$, we have by Lemma \ref{lem:degree} that
\begin{equation*}
	d_{S_{h}}(H(\cdot,\alpha), \tilde{S}_{h}, 0)\quad 
	\textrm{is independent of $\alpha \in [0,1]$}.
\end{equation*}
The proof will be completed by proving that $d_{S_{h}}(H(\cdot,\alpha=0), \tilde{S}_{h}, 0) \neq 0$. 
To see this, observe that the problem $H(\vrho_{h}, \vc{u}_{h}, \vc{w}_{h},0) = 0$ is
equivalent to finding a triplet $(\vrho_{h}, \vc{u}_{h},\vc{w}_{h}) \in \tilde{S}_{h}$ satisfying
\begin{equation}\label{eq:linsys1}
	\int_{\Omega}\vrho_{h}\phi_{h}\ dx 
	= \int_{\Omega} \vrho^{k-1}_{h}\phi_{h}\ dx , 
	\quad \forall \phi_{h} \in Q_{h}(\Omega),
\end{equation}
and
\begin{equation}\label{eq:linsys2}
	\begin{split}
		\int_{\Omega}\mu \Curl \vc{w}_{h} \vc{v}_{h} 
		+ (\lambda + \mu) \Div \vc{u}_{h} \Div \vc{v}_{h}\ dx &
		= \int_{\Omega} \vc{f}_{h}\vc{v}_{h}\ dx, 
		\quad \forall \vc{v}_{h} \in \vc{V}_{h}(\Omega), \\
		\int_{\Omega}\Curl \vc{\eta}_{h}\vc{u}_{h} 
		- \vc{w}_{h}\vc{\eta}_{h}\ dx &= 0, 
		\quad \forall \vc{\eta}_{h} \in \vc{W}_{h}(\Omega).
	\end{split}
\end{equation}
Clearly, \eqref{eq:linsys1} has the solution $\vrho_{h} = \vrho_{h}^{k-1}$. 
Moreover, \eqref{eq:linsys2} 
is a system on mixed form admitting a unique solution provided that 
the finite element spaces satisfy the Babsuka--Brezzi condition. 
However, the commuting diagram property satisfied by our finite element spaces
immediately renders the Babuska--Brezzi condition 
satisfied (cf.~Theorem \ref{theorem:mixedelliptic}). 
\end{proof}


\section{Basic estimates}\label{sec:basic-est}
In this section we establish a few estimates to be used later on, including 
square-integrability of the pressure and weak time-continuity of the density. 
However, we begin with the following lemma providing us with 
a renormalized formulation of the continuity scheme \eqref{FEM:contequation}.

\begin{lemma}[Renormalized continuity scheme]
Fix any $m=1,\ldots,M$ and let $(\vrho_{h}^m, \vc{u}_{h}^{m}) \in Q_{h} \times \vc{V}_{h}$ satisfy the continuity scheme \eqref{FEM:contequation}. 
Then $(\vrho_{h}^m, \vc{u}_{h}^{m})$ also satisfies 
the renormalized continuity scheme
\begin{equation}\label{FEM:renormalized}
	\begin{split}
		&\int_{\Omega}B(\vrho_{h}^m)\phi_{h}\ dx \\
		&\qquad - \Delta t\sum_{\Gamma \in \Gamma^I_{h}} \int_\Gamma \left(B(\vrho^m_-)(\vc{u}^{m}_h \cdot \nu)^+ 
		+ B(\vrho^m_+)(\vc{u}^{m}_h \cdot \nu)^-\right) \jump{\phi_{h}}_\Gamma\ dx \\
		&\qquad + \Delta t \int_{\Omega}b(\vrho_{h}^m)\Div \vc{u}^{m}_{h}\phi_{h}\ dx 
		+  \int_{\Omega}B''(\xi(\vrho_{h}^m, \vrho_{h}^{m-1})) \jump{\vrho_{h}^{m-1}}^2\phi_{h}\ dx  \\
		&\qquad  
		+\Delta t\sum_{\Gamma \in \Gamma^I_{h}}
		\int_{\Gamma}B''(\xi^\Gamma(\vrho^m_{+},\vrho^m_{-}))
		\jump{\vrho^m_{h}}^2_{\Gamma}(\phi_{h})_{-}(\vc{u}_{h}^{m} \cdot \nu)^+ \\
		&\qquad \qquad \qquad \qquad -B''(\xi^\Gamma(\vrho^m_{-},
		\vrho^m_{+}))\jump{\vrho^m_{h}}^2_{\Gamma}(\phi_{h})_{+}(\vc{u}_{h}^{m} \cdot \nu)^-\ dS(x) \\
		&= \int_{\Omega}B(\vrho_{h}^{m-1})\phi_{h}\ dx,
		\qquad \forall \phi_{h} \in Q_{h}(\Omega),
	\end{split}
\end{equation} 
for any $B\in C[0,\infty)\cap C^2(0,\infty)$ with $B(0)=0$ 
and $b(\vrho) := \vrho B'(\vrho) - B(\vrho)$. Given two positive real numbers $a_1$ and $a_2$, we denote by $\xi(a_1,a_2)$ and $\xi^\Gamma(a_1,a_2)$ two numbers 
between $a_1$ and $a_2$ (they will be precisely defined below).
\end{lemma}

\begin{proof}
Since $x\mapsto B'(\vrho_{h}^m(x))\phi_{h}(x)$ is piecewise constant, we 
can take $B'(\vrho_{h}^m)\phi_{h}$ as a test function in 
the continuity scheme \eqref{FEM:contequation-newform}, yielding
\begin{equation}\label{renorm-mid1}
	\begin{split}
		&\int_{\Omega}\jump{\vrho_{h}^{m-1}}B'(\vrho_{h}^m)\phi_{h}\ dx \\
		&= -\Delta t \sum_{E \in E_{h}}\int_{\binner }\left(\vrho_{+}^m(\vc{u}^{m}_{h}\cdot \nu)^+ 
		+ \vrho_{-}^m(\vc{u}^m_{h} \cdot \nu)^-\right)B'(\vrho_{+}^m)\phi_{h}\ dS(x) \\
		& = -\Delta t \sum_{E \in E_{h}}\int_{\binner}
		\left(B'(\vrho^m_{+})\vrho^m_{+}(\vc{u}^{m}_{h}\cdot \nu) 
		- \jump{\vrho^m_{h}}_{\partial E}B'(\vrho^m_{+})(\vc{u}^{m}_{h} \cdot \nu)^-
		\right)\phi_{h} \ dS(x).
	\end{split}
\end{equation}
A Taylor expansion yields
\begin{equation*}
	B'(z)(y-z)=B(y)-B(z)-B''(z^*)(y-z)^2,
\end{equation*}
for some number $z^*$ between $z$ and $y$. Consequently,
\begin{equation*}
	\begin{split}
		\jump{\vrho^m_{h}}_{\partial E} B'(\vrho_{+}^m) &= \jump{B(\vrho^m_{h})}_{\partial E} 
		-B''(\xi^{\partial E}(\vrho_{+}^m, \vrho_{-}^m))\jump{\vrho^m_{h}}_{\partial E}^2, \\
		\jump{\vrho^{m-1}_{h}} B'(\vrho_{h}^{m}) &= \jump{B(\vrho^{m-1}_{h})}
		- B''(\xi(\vrho_{h}^m, \vrho_{h}^{m-1}))\jump{\vrho^{m-1}_{h}}^2,
	\end{split}
\end{equation*}
where 
$$
\xi^{\partial E}(\vrho_{+}^m, \vrho_{-}^m)(x) 
\in [\vrho_{-}^m(x), \vrho^m_{+}(x)], \qquad x \in \partial E,
$$ 
and
$$
\xi(\vrho_{h}^m, 
\vrho_{h}^{m-1})(x) \in [\vrho_{h}^{m-1}(x),
\vrho_{h}^m(x)], \qquad x \in \Omega.
$$
  
Inserting these identities in \eqref{renorm-mid1}, recalling the definition of $b$, 
and applying Green's theorem, we achieve
\begin{equation}\label{renorm-mid2}
	\begin{split}
		&\int_{\Omega}\left(B(\vrho_{h}^m) - B(\vrho_{h}^{m-1})\right)\phi_{h}\ dx
		+\int_{\Omega}
		B''(\xi(\vrho_{h}^m,\vrho_{h}^{m-1}))
		\jump{\vrho_{h}^{m-1}}^2\phi_{h}\ dx \\
		& =-\Delta t\int_{\Omega}b(\vrho_{+}^m)\Div \vc{u}^{m}_{h}\phi_{h}\ dx \\
		&\quad -\Delta t\sum_{E \in E_{h}}\int_{\binner}
		\left(B(\vrho^m_{+})(\vc{u}_{h}^{m} \cdot \nu)^+
		+ B(\vrho^m_{-})(\vc{u}_{h}^{m} \cdot \nu)^-\right)\phi_{h}\ dS(x) \\
		&\quad + \Delta t\sum_{E \in E_{h}}\int_{\binner}
		B''(\xi^{\partial E}(\vrho_{+}^m, \vrho_{-}^m))
		\jump{\vrho^m_{h}}^2_{\partial E}(\vc{u}_{h}^{m} \cdot \nu)^-\phi_{h}\ dS(x).
	\end{split}
\end{equation}
Denote by $I$ the second term on the right-hand side of the equality sign. 
Then, as in Remark \ref{rem:E-VS-Gamma}, we have the identity
\begin{equation}\label{eq:renorm-id1}
	\begin{split}
		I= \Delta t\sum_{\Gamma \in \Gamma^I_{h}}\int_{\Gamma}
		\left(B(\vrho^m_{+})(\vc{u}_{h}^{m} \cdot \nu)^+
		+ B(\vrho^m_{-})(\vc{u}_{h}^{m} \cdot \nu)^-\right)\jump{\phi_{h}}_{\Gamma}\ dS(x).
	\end{split}
\end{equation}

Recalling that for each $\Gamma \in \Gamma_{h}$ we let $E_{+}$ and $E_{-}$ be 
the two elements sharing the face $\Gamma$ and such that the normal
component associated with $\Gamma$ points from $E_{+}$ to $E_{-}$, we can write
\begin{equation}\label{eq:renorm-id2}
	\begin{split}
		&\Delta t\sum_{E \in E_{h}}\int_{\binner}
		B''(\xi^{\partial E}(\vrho_{+}^m, \vrho_{-}^m))
		\jump{\vrho^m_{h}}^2_{\partial E}(\vc{u}_{h}^{m} \cdot \nu)^-\phi_{h}\ dS(x) \\
		& = \Delta t\sum_{\Gamma \in \Gamma^I_{h}}\int_{\Gamma}
		-B''(\xi^{\partial E_{-}}(\vrho^m_{+},
		\vrho_{-}^m))\jump{\vrho^m_{h}}^2_{\Gamma}(\phi_{h})_{-}(\vc{u}_{h}^{m} \cdot \nu)^+ \\
		&\qquad\qquad\qquad\quad
		+B''(\xi^{\partial E_{+}}(\vrho^m_{+},\vrho_{-}^m))
		\jump{\vrho^m_{h}}^2_{\Gamma}(\phi_{h})_{+}(\vc{u}_{h}^{m}\cdot\nu)^-\ dS(x).
	\end{split}
\end{equation}

Once we introduce into \eqref{eq:renorm-id2} the notations
$$
\xi^\Gamma (\vrho^m_{+}, \vrho^m_{-}) 
:= \xi^{\partial E_{+}}(\vrho^m_{+}, \vrho_{-}^m), 
\qquad \xi^\Gamma (\vrho^m_{-}, \vrho^m_{+}) 
:= \xi^{\partial E_{-}}(\vrho^m_{+}, \vrho_{-}^m),
$$
inserting \eqref{eq:renorm-id1}, \eqref{eq:renorm-id2} into \eqref{renorm-mid2} 
yields the final result \eqref{FEM:renormalized}. 
\end{proof}

In what follows we shall need a discrete Hodge decomposition. 
The following lemma is a consequence of \eqref{eq:Vh-decomp}.

\begin{lemma}\label{lemma:hodge}
\solutiontext For each fixed $h>0$, there exist unique 
functions $\vc{\zeta}_{h}^m \in \vc{W}_{h}^{0,\perp}$ and 
$\vc{z}_{h}^m \in \vc{V}_{h}^{0,\perp}$ such that
\begin{equation}\label{eq:hodge-disc-time}
	\vc{u}_{h}^m = \Curl \vc{\zeta}_{h}^m 
	+ \vc{z}_{h}^m, \qquad m=1,\ldots,M.
\end{equation}
Moreover, if we let $\vc{\zeta}_{h}(t,x)$, $\vc{z}_{h}(t,x)$ denote 
the functions obtained by extending, as in \eqref{eq:num-scheme-II}, 
$\{\vc{\zeta}_{h}^m\}_{m=1}^M$, $\{\vc{z}_{h}^m\}_{m=1}^M$ 
to the whole of $(0,T] \times \Omega$, then
\begin{equation*}
	\vc{u}_{h}(t,\cdot) = \Curl \vc{\zeta}_{h}(\cdot,t)
	+\vc{z}_{h}(\cdot,t), \qquad t \in (0,T).
\end{equation*}
\end{lemma}

We now state a basic stability estimate 
satisfied by any solution of the discrete problem 
given in Definition \ref {def:num-scheme}. 

\begin{lemma}\label{lemma:stability} 
\solutiontext For $\vrho> 0$, set 
$
P(\vrho):=
	\frac{a}{\gamma-1}\vrho^\gamma.
$
For any $m=1,\dots,M$, we have 
\begin{equation}\label{eq:stabilityincerrors}
	\begin{split}
		&\int_{\Omega} P(\vrho_{h}^m)\ dx \\
		& \quad 
		+ \sum_{k=1}^{m}\int_{\Omega} P''(\xi(\vrho_{h}^k,
		\vrho_{h}^{k-1}))\jump{\vrho^{k-1}_{h}}^2\ dx \\
		&\quad\qquad+ \sum_{k=1}^{m}\sum_{\Gamma \in \Gamma^I_{h}}\Delta t
		\int_{\Gamma}P''(\vrho^k_{\dagger})\jump{\vrho^k_{h}}_{\Gamma}^2
		\abs{\vc{u}^k_{h} \cdot \nu}\ dx \\
		& \quad
		+\sum_{k=1}^m\Delta t \int_{\Omega} \abs{\vc{u}^k_{h}}^2\ dx
		+\sum_{k=1}^m\Delta t\int_{\Omega} \abs{\Div \vc{u}^k_{h}}^2\ dx
		\\ & \quad\qquad
		+\sum_{k=1}^m\Delta t\int_{\Omega} \abs{\vc{w}^k_{h}}^2 \ dx 
		+\sum_{k=1}^m\Delta t\int_{\Omega} \abs{\Curl \vc{w}^k_{h}}^2\ dx \\
		& \leq \int_{\Omega} P(\vrho_{0}) \ dx 
		+C \sum_{k=1}^m\Delta t\int_{\Omega} \abs{\vc{f}^k_{h}}^2 \ dx.
	\end{split}
\end{equation}
Consequently, 
$
\vrho_{h}\inb L^\infty(0,T;L^\gamma(\Omega)).
$ 
\end{lemma}

\begin{proof}
Since $P'(\vrho)\vrho-P(\vrho)=p(\vrho)$ and $\vrho_h>0$, it follows 
by taking $\phi_{h}\equiv 1$ in the renormalized 
scheme \eqref{FEM:renormalized} that
\begin{equation}\label{eq:stabilityeq1}
	\begin{split}
		& \int_{\Omega}P(\vrho_{h}^k)\ dx
		+ \Delta t \int_{\Omega}p(\vrho_{h}^k)\Div \vc{u}^{k}_{h}\ dx 
		+\int_{\Omega}P''(\xi(\vrho_{h}^k,
		\vrho_{h}^{k-1}))\jump{\vrho_{h}^{k-1}}^2\ dx \\
		&\qquad\qquad +\Delta t\sum_{\Gamma \in \Gamma^I_{h}}
		\int_{\Gamma}P''(\xi^\Gamma (\vrho^m_{+},
		\vrho^m_{-}))\jump{\vrho^k_{h}}^2_{\Gamma}(\vc{u}_{h}^{k} \cdot \nu)^+ 
		\\ &\qquad \qquad\qquad 
		-P''(\xi^\Gamma (\vrho^m_{-},
		\vrho^m_{+}))\jump{\vrho^k_{h}}^2_{\Gamma}(\vc{u}_{h}^{k} \cdot \nu)^-\ dS(x) 
		= \int_{\Omega}P(\vrho^{k-1})\ dx.
	\end{split}
\end{equation}
For $k=1,\ldots,M$ and $x\in\bigcup_{\Gamma \in \Gamma^I_{h}}\Gamma$, set
\begin{equation}\label{eq:vrhomin}
	\vrho_{\dagger}^k(x):=
	\begin{cases}
		\max\{\vrho_{+}^k(x),\vrho_{-}^k(x)\}, & 1 < \gamma \leq 2, \\
		\min\{\vrho_{+}^k(x),\vrho_{-}^k(x)\}, & \gamma \ge 2,
	\end{cases}
\end{equation}
and note that 
\begin{equation}\label{eq:doublederiv}
	\begin{split}
		&\Delta t\sum_{\Gamma \in \Gamma^I_{h}}
		\int_{\Gamma}P''(\xi^\Gamma (\vrho^m_{+},\vrho^m_{-}))
		\jump{\vrho^k_{h}}^2_{\Gamma}(\vc{u}_{h}^{k} \cdot \nu)^+ 
		\\ & \qquad\qquad\qquad 
		-P''(\xi^\Gamma (\vrho^m_{-},\vrho^m_{+}))
		\jump{\vrho^k_{h}}^2_{\Gamma}(\vc{u}_{h}^{k} 
		\cdot \nu)^-\ dS(x), \\
		& \qquad
		\geq \Delta t \sum_{\Gamma \in
		\Gamma^I_{h}}\int_{\Gamma}P''(\vrho_{\dagger}^k)
		\jump{\vrho^k_{h}}^2_{\Gamma}\abs{\vc{u}_{h}^{k} 
		\cdot \nu}\ dS(x).
	\end{split}
\end{equation}

Next, by using $\vc{v}_{h}=\vc{u}^k_{h}$ as a test 
function in the first equation of \eqref{FEM:momentumeq} 
and then using the second equation of \eqref{FEM:momentumeq} with 
$\vc{\eta}_{h}= \vc{w}^k_{h}$, we obtain the identity
\begin{equation}\label{eq:stabilityeq2}
	\int_{\Omega} p(\vrho_{h}^k)\Div \vc{u}_{h}^{k}\ dx 
	= (\mu+\lambda)\int_{\Omega} \abs{\Div \vc{u}_{h}^k}^2\ dx 
	+ \mu \int_{\Omega}\abs{\vc{w}_{h}^k}^2\ dx 
	- \int_{\Omega}\vc{f}_{h}^k\vc{u}^k_{h}\ dx.
\end{equation}
Similarly, specifying $\vc{v}_{h}=\Curl \vc{w}^k_{h}$ in the 
first equation of \eqref{FEM:momentumeq} yields
\begin{equation*}
	\mu\int_{\Omega}\abs{\Curl \vc{w}^k_{h}}^2\ dx 
	=\int_{\Omega}\vc{f}_{h}^k\Curl \vc{w}^k_{h}\ dx.
\end{equation*}
An application of Cauchy's inequality (with epsilon) then yields
\begin{equation}\label{eq:stabilityeq3}
	\int_{\Omega}\abs{\Curl\vc{w}^k_{h}}^2\ dx
	\leq C\int_{\Omega}\abs{\vc{f}_{h}^k}^2\ dx.
\end{equation}

Thanks to \eqref{eq:hodge-disc-time}, we can write 
$\vc{u}_{h}^k=\Curl \vc{\zeta}_{h}^k+\vc{z}_{h}^k$ for 
with $\vc{\zeta}_{h}^k \in \vc{W}^{0,\perp}_{h}$ and 
$\vc{z}^k_{h} \in  \vc{V}^{0, \perp}_{h}$. 
Choosing $\vc{\eta}_{h}=\Curl \vc{\zeta}_{h}^k$ in 
the second equation of \eqref{FEM:momentumeq} gives
\begin{equation*}
	\int_{\Omega}\abs{\Curl \vc{\zeta}_{h}^k}^2\ dx
	=\int_{\Omega}\vc{w}^k_{h}\vc{\zeta}_{h}^k\ dx
	\le \left(\int_{\Omega}\abs{\vc{w}^k_{h}}^2\ dx\right)^\frac{1}{2}
	\left(\int_{\Omega}\abs{\vc{\zeta}_{h}^k}^2\ dx\right)^\frac{1}{2}.
\end{equation*}
Thus, since the discrete Poincar\'e 
inequality \eqref{Poincare2} tells us that
\begin{equation*}
	\int_{\Omega}\abs{\vc{\zeta}_{h}^k}^2\ dx 
	\leq C\int_{\Omega}\abs{\Curl \vc{\zeta}_{h}^k}^2\ dx,
\end{equation*} 
we arrive at the estimate
\begin{equation}
\label{eq:curlcomponent}
	\int_{\Omega}\abs{\vc{\zeta}_{h}^k}^2
	+ \abs{\Curl \vc{\zeta}_{h}^k}^2\ dx
	\leq C\int_{\Omega}\abs{\vc{w}^k_{h}}^2\ dx.
\end{equation}
In view of the discrete Poincar\'e inequality \eqref{Poincare1}, we also have 
\begin{equation*}
	\int_{\Omega}\abs{\vc{z}^k_{h}}^2\ dx
	\leq C \int_{\Omega}\abs{\Div \vc{u}^k_{h}}^2\ dx,
\end{equation*}
which, together with \eqref{eq:curlcomponent}, allow us to conclude
\begin{equation}\label{eq:stabilityeq4}
	\int_{\Omega}\abs{\vc{u}^k_{h}}^2\ dx
	=\int_{\Omega}\abs{\vc{\zeta}_{h}^k}^2
	+\abs{\vc{z}_{h}^k}^2\ dx 
	\leq C\left(\int_{\Omega}\abs{\vc{w}^k_{h}}^2\ dx
	+\int_{\Omega}\abs{\Div \vc{u}_{h}^k}^2\ dx\right).
\end{equation}

Now, by first inserting \eqref{eq:stabilityeq2} 
into \eqref{eq:stabilityeq1} and subsequently 
utilizing \eqref{eq:doublederiv}, \eqref{eq:stabilityeq3}, \eqref{eq:stabilityeq4} 
and Cauchy's inequality (with epsilon) to 
treat the last integral appearing in \eqref{eq:stabilityeq2}, we 
acquire the estimate
\begin{equation}\label{eq:lastinstab}
	\begin{split}
		&\int_{\Omega}P(\vrho_{h}^k) - P(\vrho_{h}^{k-1})\ dx \\
		&\quad +\int_{\Omega}P''(\xi(\vrho_{h}^k, \vrho_{h}^{k-1}))
		\jump{\vrho_{h}^{k-1}}^2 \ dx 
		+\Delta t \sum_{\Gamma \in \Gamma^I_{h}}\int_{\Gamma}P''(\vrho_{\dagger}^k)
		\jump{\vrho^k_{h}}^2_{\Gamma}\abs{\vc{u}_{h}^{k} \cdot \nu}\ dS(x) \\
		&\quad\quad + \Delta t \int_{\Omega}\abs{\vc{u}_{h}^k}^2\ dx 
		+ \Delta t\int_{\Omega}\abs{\Div \vc{u}_{h}^k}^2\ dx \\
		&\quad \quad \quad 
		+\Delta t \int_{\Omega}\abs{\vc{w}_{h}^k}^2\ dx 
		+\Delta t \int_{\Omega}\abs{\Curl \vc{w}_{h}^k}^2\ dx 
		\leq C\Delta t \int_{\Omega}\abs{\vc{f}_{h}^k}^2\ dx.
	\end{split}
\end{equation}
Finally, summing \eqref{eq:lastinstab} over $k$ we 
conclude that \eqref{eq:stabilityincerrors} holds.
\end{proof} 

The stability estimate only provides the 
bound $p(\vrho_{h}) \in_{b} L^\infty(0,T;L^1(\Omega))$.
Hence, it is not clear that $p(\vrho_{h})$ converges 
weakly to an integrable function. Moreover, 
the subsequent analysis relies heavily on the pressure 
having higher integrability. In the ensuing 
lemma we establish that the pressure is in fact bounded 
in $L^2(0,T;L^2(\Omega))$, independently of $h$.

To simplify notation, we denote the effective viscous flux by
\begin{equation}\label{eq:visc-flux}
	\eff(\vrho_{h}, \vc{u}_{h}) 
	= p(\vrho_{h})-(\mu + \lambda)\Div \vc{u}_{h}.
\end{equation}
We will also continue to use this notation 
in the subsequent sections.

\begin{lemma}[Higher integrability on the pressure]\label{lemma:higherorderpressure} 
\solutiontext Then 
$$
p(\vrho_{h}) \in_{b} L^{2}((0,T) \times \Omega).
$$
\end{lemma}

\begin{proof}
For all $m=1,\ldots,M$, let 
$\vc{v}_{h}^m \in \vc{V}_{h}^{0, \perp}$ be such that
\begin{equation*}
	\Div \vc{v}^m_{h} = \eff(\vrho^m_{h}, \vc{u}^m_{h})
	-\frac{1}{\abs{\Omega}}
	\int_{\Omega}\eff(\vrho^m_{h},\vc{u}^m_{h})\ dx.
\end{equation*} 
Now, since the momentum scheme \eqref{FEM:momentumeq} gives
\begin{equation*}
		\int_{\Omega}\eff(\vrho_{h}^m, \vc{u}_{h}^m)\Div \vc{v}_{h}\ dx 
		= -\int_{\Omega}\vc{f}_{h}^m \vc{v}_{h}\ dx, 
		\quad \forall \vc{v}_{h}\in \vc{V}_{h}^{0,\perp},
\end{equation*}
we can use $\vc{v}_{h}^m$ as test function to obtain
\begin{align*}
	\int_{\Omega}\abs{\eff(\vrho^m_{h}, \vc{u}^m_{h})}^2\ dx 
	=\frac{1}{\abs{\Omega}}\left(\int_{\Omega} \eff(\vrho_{h}^m,\vc{u}_{h}^m) \ dx\right)^2
	- \int_{\Omega}\vc{f}_{h}^m\vc{v}_{h}^m\ dx.
\end{align*}
Hence, with $\epsilon>0$,
\begin{equation}\label{eq:higher-epsilon}
	\begin{split}
		&\int_{\Omega}\abs{\eff(\vrho_{h}^m, \vc{u}_{h}^m)
		-\frac{1}{\abs{\Omega}}\int_{\Omega}\eff(\vrho_{h}^m,\vc{u}_{h}^m)\ dx }^2\ dx \\
		& \quad 
		=-\int_{\Omega}\vc{f}_{h}^m\vc{v}_{h}^m\ dx 
		\leq \frac{1}{4\epsilon}\int_{\Omega}\abs{\vc{f}_{h}^m}^2 \ dx
		+\epsilon\int_{\Omega}\abs{\vc{v}_{h}^m}^2\ dx.
	\end{split}
\end{equation}

By the Poincar\'e inequality \eqref{Poincare1}
$$
\int_{\Omega}\abs{\vc{v}_{h}^m}^2\ dx 
\leq 
C\int_{\Omega}\abs{\eff(\vrho_{h}^m, \vc{u}_{h}^m) 
- \frac{1}{\abs{\Omega}}
\int_{\Omega}\eff(\vrho_{h}^m,\vc{u}_{h}^m) \ dx }^2\ dx.
$$
Consequently, by fixing $\epsilon$ small enough in 
\eqref{eq:higher-epsilon},
\begin{equation*}
	\int_{\Omega}\abs{\eff(\vrho_{h}^m, \vc{u}_{h}^m) 
	- \frac{1}{\abs{\Omega}}\int_{\Omega}
	\eff(\vrho_{h}^m,\vc{u}_{h}^m) \ dx }^2\ dx 
	\leq C\int_{\Omega}\abs{\vc{f}_{h}^m}^2\ dx,
\end{equation*}
and thus
\begin{equation}\label{eq:fluxbound1}
	\int_{\Omega}\abs{\eff(\vrho_{h}^m,\vc{u}_{h}^m)}^2\ dx 
	\leq 
	\frac{1}{\abs{\Omega}}\left(\int_{\Omega} 
	\eff(\vrho_{h}^m, \vc{u}_{h}^m) \ dx\right)^2 
	+ C\int_{\Omega}\abs{\vc{f}_{h}^m}^2\ dx.
\end{equation}

Now, due to the boundary conditions,
\begin{equation*}
	\int_{\Omega}\eff(\vrho_{h}^m,\vc{u}_{h}^m)\ dx 
	= \int_{\Omega}p(\vrho_{h}^m) \ dx \leq C,
\end{equation*}
where we also have put into use Lemma \ref{lemma:stability} and subsequently our 
assumptions on the source term $\vc{f}$ and the initial data $\vrho_0$.
Hence, \eqref{eq:fluxbound1} allows us to conclude 
\begin{equation*}
	\int_{\Omega}\abs{\eff(\vrho_{h}^m, \vc{u}_{h}^m)}^2\ dx 
	\leq C\left(1 + \int_{\Omega}\abs{\vc{f}_{h}^m}^2\ dx\right), 
	\qquad m=1, \ldots, M,
\end{equation*}
from which we obtain
\begin{equation*}
	\sum_{m=1}^M\Delta t \int_{\Omega}
	\abs{\eff(\vrho_{h}^m, \vc{u}_{h}^m)}^2\ dx 
	\leq C\left(T + \sum_{m=1}^M \Delta t 
	\int_{\Omega}\abs{\vc{f}_{h}^m}^2\ dx\right).
\end{equation*}
In view of the definition of $\eff$, this immediately yields 
\begin{align*}
	&\sum_{m=1}^M\Delta t \int_{\Omega} 
	\abs{p(\vrho_{h})}^{2}\ dx 
	\\ & \qquad 
	\leq C\left( T + \sum_{m=1}^M\Delta t 
	\int_{\Omega}\abs{\Div \vc{u}_{h}^m}^2\ dx 
	+ \sum_{m=1}^M \Delta t \int_{\Omega}
	\abs{\vc{f}_{h}^m}^2\ dx \right),
\end{align*}
which, due to Lemma \ref{lemma:stability} 
and $\vc{f}\in L^2((0,T)\times \Om)$, concludes the proof. 
\end{proof}

We conclude this section by establishing a weak time continuity of the density approximation. 
For this purpose we shall need the following technical lemma, which provides 
a bound on the artificial diffusion introduced by the upwind discretization 
of the continuity equation.

\begin{lemma} \label{lemma:productionbound} 
\solutiontext There exists a constant $C>0$, 
depending only on the shape regularity of $E_h$, 
the size of $|\Om|$, and the final time $T$, such that 
\begin{equation}\label{eq:productionerror}
	\begin{split}
		&\left| \sum_{E \in E_{h}}\int_{0}^T 
		\int_{\binner}\jump{\vrho_{h}}_{\partial E}(\vc{u}_{h} \cdot \nu)^-
		(\Pi_h^Q\phi - \phi)\ dS(x)dt\right|
		\\ &  \qquad \leq 
		C
		\norm{D\phi}_{L^\infty(0,T;\vc{L}^\infty(\Omega))}
		h^\frac{1}{2},
	\end{split}
\end{equation}
for any $\phi \in L^\infty(0,T;W^{1,\infty}(\Omega))$.
\end{lemma}

\begin{proof}
We shall need the auxiliary function
$$
B(z) = 
\begin{cases}
	z^2,& \gamma > 2, \\
	z^\gamma, & \gamma \leq 2.
\end{cases}
$$
Moreover, set 
\begin{equation}\label{eq:phim-phimh}
	\phi^m(x) = \frac{1}{\Delta t}
	\int_{t^{m-1}}^{t^m}\phi(s,x)\ ds, 
	\quad 
	\phi_{h}^m = \Pi_{h}^Q \phi^m, 
	\quad  m=1,\ldots,M.
\end{equation}
Using $B''(z) > 0$ for $z>0$ and H\"older's inequality, we obtain 
\begin{align*}
		I^2 & :=\left|\sum_{m=1}^M\sum_{E \in E_{h}}
		\Delta t \int_{\binner}\jump{\vrho^m_{h}}_{\partial E}
		(\vc{u}^m_{h} \cdot \nu)^-
		(\phi^m_{h} - \phi^m)\ dS(x) \right|^2 \\ 
		& \leq \left(\sum_{m=1}^M 
		\sum_{E \in E_{h}}\Delta t
		\int_{\binner}B''(\vrho^m_{\dagger})
		\jump{\vrho^m_{h}}^2\abs{\vc{u}^m_{h} \cdot \nu}
		\ dS(x)\right) \\
		& \qquad \times \left(\sum_{m=1}^M 
		\sum_{E \in E_{h}}\Delta t\int_{\binner}
		\left(B''(\vrho^m_{\dagger})\right)^{-1}\abs{\vc{u}^m_{h} \cdot \nu}
		\abs{\phi^m_{h} - \phi^m}^2\ dS(x)  \right),
		\\ & =: I_1\times I_2,
\end{align*}
where the ``intermediate" numbers $\vrho_{\dagger}^m$ are 
defined in \eqref{eq:vrhomin}.

If $1 < \gamma \leq 2$, then Lemma \ref{lemma:stability} can be applied:
\begin{equation}\label{eq:contconveq1}
	I_1\leq C \int_{\Omega}B(\vrho_{0})\ dx
	+\sum_{m=1}^M \Delta t \int_{\Omega}
	\abs{\vc{f}_{h}^m}^2\ dx.
\end{equation}
However, \eqref{eq:contconveq1} continues to hold in the 
case $\gamma \geq 2$. This follows directly from 
the renormalized scheme \eqref{FEM:renormalized}, with $\phi_{h}:= 1$, 
together with the fact that 
\begin{align*}
	&\sum_{m=1}^M \Delta t \int_{\Omega}b(\vrho^m_{h}) 
	\Div \vc{u}^m_{h}\ dxdt \\
	&\qquad \leq \left(\sum_{m=1}^M \Delta t
	\int_{\Omega} \abs{\vrho_{h}^m}^4\ dx \right)^\frac{1}{2}
	\left(\sum_{m=1}^M 
	\Delta t\int_{\Omega} \abs{\Div \vc{u}_{h}^m}^2\ dx \right)^\frac{1}{2},
\end{align*}
which is bounded by Lemmas \ref{lemma:stability} 
and \ref{lemma:higherorderpressure}.

Next, using Lemma \ref{lemma:interpolation}, we have that
\begin{equation}\label{eq:prop1}
	\begin{split}
		I_2 & \leq h^2 \norm{D\phi}_{L^\infty(0,T;\vc{L}^\infty(\Omega))}^2
		\left(\sum_{m=1}^M\sum_{E \in E_{h}} \Delta t
		\int_{\binner}\left[\left(B''(\vrho^m_{\dagger})\right)^{-1}\right]^2
		\ dS(x)\right)^\frac{1}{2}
		\\ &\qquad\qquad\qquad\qquad\qquad\qquad\quad\times 
		\left(\sum_{m=1}^M\sum_{E \in E_{h}} \Delta t\int_{\binner}
		\abs{\vc{u}^m_{h} \cdot \nu}^2\ dS(x)\right)^\frac{1}{2}.
	\end{split}
\end{equation}
Thanks to Lemma \ref{lemma:edgebounds} and Lemma \ref{lemma:inverse},
\begin{equation}\label{eq:prop2}
	\int_{\partial E}\abs{\vc{u}^m_{h} \cdot \nu}^2\ dS(x) 
	\leq c\, h^{-1}\int_{E}\abs{\vc{u}_{h}^m}^2\ dx.
\end{equation}
Moreover, since
$$
\left(B''(\vrho_{\dagger}^m)\right)^{-1} 
\leq \abs{\vrho^m_{+}+ \vrho^m_{-}}^{2-\gamma} 
\leq C(1+ \vrho^m_{+} +\vrho^m_{-}),
$$
whenever $1< \gamma \le 2$, and 
$\left(B''(\vrho^m_{\dagger})\right)^{-1}
=\frac{1}{2}$, whenever $\gamma>2$, 
Lemma \ref{lemma:edgebounds} also gives
$$
\int_{\partial E} \left[\left(B''(\vrho^m_{\dagger})\right)^{-1}\right]^2\ dS(x) 
\leq Ch^{-1}\left(|E|+ \int_{E\,\cup\,\mathcal{N}(E)}
\abs{\vrho_{h}^m}^2\ dx\right),
$$
where $\mathcal{N}(E)$ denotes the union of the neighboring elements of $E$. 
Observe that
\begin{equation}\label{eq:prop3}
	\begin{split}
		&\sum_{E \in E_{h}} h^{-1}\left(|E|+\sum_{F \in \mathcal{N}(E)}
		\int_{F}\abs{\vrho_{h}^m}^2\ dx\right) \\
		& \qquad 
		\leq h^{-1}\left(|\Omega| 
		+ \int\limits_{\bigcup\limits_{E \in E_{h}} 
		\left(E\,\cup\,\mathcal{N}(E)\right)}
		|\vrho_{h}^m|^2\ dx\right) \\
		& \qquad\leq h^{-1}\left(|\Omega| + (N+2)\int_{\Omega}
		|\vrho_{h}^m|^2\ dx\right),
	\end{split}
\end{equation}
where we have utilized the fact that 
$$
\abs{\bigcup\limits_{E \in E_{h}} 
\left(E\,\cup\, \mathcal{N}(E)\right)} 
\leq (N+2)|\Omega|,
$$
which is true since the maximal cardinality of the 
set $\Set{(\mathcal{N}(E)\cup E)\cap F}_{E \in E_{h}}$ is 
$N+2$ for any $F \in E_{h}$. Inserting \eqref{eq:prop2} and \eqref{eq:prop3} 
into \eqref{eq:prop1}, we have arrived at
\begin{align*}
	I_2 &\leq h^2 \norm{D\phi}_{L^\infty(0,T;\vc{L}^\infty(\Omega))}^2
	h^{-\frac{1}{2}}h^{-\frac{1}{2}} 
	\\ &\qquad \quad\times 
	\left(T+\sum_{m=1}^M \Delta t \int_{\Omega} \abs{\vrho_{h}^m}^2\ dx\right)^\frac{1}{2} 
	\left(\sum_{m=1}^M \Delta t\int_{\Omega}\abs{\vc{u}_{h}^m}^2\ dx\right)^\frac{1}{2}
	\\& \leq  C\, h \norm{D\phi}_{L^\infty(0,T;\vc{L}^\infty(\Omega))}^2,
\end{align*}
where Lemmas \ref{lemma:stability} and \ref{lemma:higherorderpressure} 
have been used to work out the last inequality. 
This concludes the proof of \eqref{eq:productionerror}.
\end{proof}

To simplify the notation, let us introduce the 
interpolation operator
\begin{equation}
	\left(\Pi_{\mathcal{L}}f\right)(t) 
	= f^{m-1} + \frac{t - t^{m-1}}{\Delta t}(f^{m} - f^{m-1}), 
	\qquad t \in (t^{m-1},t^m).
\end{equation}

\begin{lemma}\label{lemma:timecont}
\solutiontext
Then
$$
\frac{\partial}{\partial t}
\left(\Pi_{\mathcal{L}}\vrho_{h}\right) 
\in_{b} L^1(0,T;W^{-1,1}(\Omega)).
$$
\end{lemma}

\begin{proof}
Fix $\phi \in L^\infty(0,T;W^{1,\infty}(\Omega))$, and recall the 
definitions of $\phi^m$, $\phi_{h}^m$, cf.~\eqref{eq:phim-phimh}. 
The continuity scheme \eqref{FEM:contequation} 
with $\phi_{h}^m$ as test function reads
\begin{equation}\label{eq:tcont1}
	\begin{split}
		& \Delta t\int_{\Omega} \frac{d}{dt}
		\left(\Pi_{\mathcal{L}}\vrho_{h}\right)\phi^m\ dxdt 
		\\ & \qquad 
		= \Delta t\sum_{\Gamma \in \Gamma^I_{h}}
		\int_\Gamma\left(\vrho^m_-(\vc{u}^{m}_h \cdot \nu)^+ 
		+\vrho^m_+(\vc{u}^{m}_h \cdot \nu)^-\right)
		\jump{\phi^m_{h}}_\Gamma\ dS(x).
	\end{split}
\end{equation}
Since the traces of $\phi^m$ taken from either side of a 
face are equal, we can write
\begin{align*}
	& \sum_{\Gamma \in \Gamma^I_{h}}
	\int_\Gamma \left(\vrho^m_-(\vc{u}^{m}_h \cdot \nu)^+ 
	+\vrho^m_+(\vc{u}^{m}_h \cdot \nu)^-\right)\jump{\phi^m_{h}}_\Gamma\ dx \\
	&=\sum_{\Gamma \in \Gamma^I_{h}}\int_{\Gamma} \left( \vrho^m_{+}(\vc{u}^m_{h} \cdot \nu)^- 
	+ \vrho^m_{-}(\vc{u}^m_{h} \cdot \nu)^+\right)\jump{\phi^m_{h} - \phi^m}\ dS(x), \\
	&= -\sum_{E \in E_{h}}\int_{\binner} \left( \vrho^m_{+}(\vc{u}^m_{h} \cdot \nu)^+ 
	+ \vrho^m_{-}(\vc{u}^m_{h} \cdot \nu)^-\right)(\phi^m_{h}- \phi^m)\ dS(x), \\
	&= \int_{\Omega}- \Div (\vrho^m_{h}\vc{u}^m_{h}(\phi^m_{h} - \phi^m))\ dx 
	+ \sum_{E \in E_{h}}\int_{\binner }\jump{\vrho^m_{h}}_{\partial E}(\vc{u}^m_{h} \cdot \nu)^-
	(\phi^m_{h} - \phi^m)\ dS(x) \\
	&= \int_{\Omega}\vrho^m_{h}\vc{u}^m_{h}\cdot D\phi^m\ dx  
	+ \sum_{E \in E_{h}}\int_{\binner}
	\jump{\vrho^m_{h}}_{\partial E}(\vc{u}^m_{h} \cdot \nu)^-
	(\phi^m_{h} - \phi^m)\ dS(x).
\end{align*}

By summing \eqref{eq:tcont1} over $m$, taking 
absolute values, and using the above identity, we find
\begin{equation*}
	\begin{split}
		&\left|\sum_{m=1}^M\Delta t \int_{\Omega} 
		\frac{d}{dt}\left(\Pi_{\mathcal{L}}
		\vrho_{h}\right)\phi^m\ dxdt\right| \\
		&\qquad \qquad \leq \left|\sum_{m=1}^M \Delta t \int_{\Omega}\vrho_{h}^m 
		\vc{u}_{h}^m D\phi^m\ dx\right| \\
		&\qquad \qquad \qquad+\left|\sum_{m=1}^M \sum_{E \in E_{h}}\Delta t
		\int_{\binner}\jump{\vrho^m_{h}}_{\partial E}(\vc{u}^m_{h} \cdot \nu)^-
		(\phi_{h}^m - \phi^m)\ dS(x)\right|.
	\end{split}
\end{equation*}
Using Lemma \ref{lemma:productionbound}, together with an 
application of H\"older's inequality, we deduce
\begin{align*}
	&\left|\sum_{m=1}^M\Delta t \int_{\Omega} 
	\frac{d}{dt}\left(\Pi_{\mathcal{L}}
	\vrho_{h}\right)\phi^m\ dx\right| \\
	&\leq \left(\sum_{m=1}^M \Delta t 
	\int_{\Omega}|\vrho_{h}^m|^2\ dx\right)^\frac{1}{2}
	\left(\sum_{m=1}^M \Delta t
	\int_{\Omega}\abs{\vc{u}_{h}^m}^2 \right)^\frac{1}{2}
	\norm{D\phi}_{L^\infty(0,T;\vc{L}^\infty(\Omega))} \\
	&\qquad \qquad + C \, h^\frac{1}{2}
	\norm{D\phi}_{L^\infty(0,T;\vc{L}^\infty(\Omega))}.
\end{align*}
By Lemmas \ref{lemma:stability} and \ref{lemma:higherorderpressure}, the first two 
factors on the right--hand side is bounded, so we conclude that
\begin{align*}
	&\left|\int_{\Delta t}^T\int_{\Omega}\frac{d}{dt}
	\left(\Pi_{\mathcal{L}}\vrho_{h}\right)\phi\ dxdt\right| \\
	&=\left|\sum_{m=1}^M\Delta t \int_{\Omega} \frac{d}{dt}
	\left(\Pi_{\mathcal{L}}\vrho_{h}\right)\phi^m\ dx\right| 
	\leq C\, (1+h^\frac{1}{2})
	\norm{D\phi}_{L^\infty(0,T;\vc{L}^\infty(\Omega))}.
\end{align*}
\end{proof}

\section{Convergence} \label{sec:conv}

\solutiontext In this section we establish that a subsequence 
of $\{\left(\vrho_{h}, \vc{w}_{h}, \vc{u}_{h}\right)\}_{h>0}$ 
converges to a weak solution of the semi--stationary Stokes system, thereby proving 
Theorem \ref{theorem:mainconvergence}. The proof is divided into 
several steps: 
\begin{enumerate}
	\item{}Convergence of the continuity scheme.
	\item{}Weak sequential continuity of the discrete viscous flux.
	\item{}Strong convergence of the density.
	\item{}Convergence of the velocity scheme.
\end{enumerate}

Our starting point is that the results of Section \ref{sec:basic-est} assure us that the approximate 
solutions $(\vc{w}_{h}, \vc{u}_{h}, \vrho_{h})$ satisfy the following 
$h$--independent bounds:
$$
\vrho_{h} \in_{b} L^\infty(0,T;L^\gamma(\Omega))
\cap L^{2\gamma}((0,T)\times \Omega)
$$
and
$$
\vc{w}_{h} \in_{b} L^2(0,T;\vc{W}^{\Curl ,2}_{0}(\Omega)), 
\quad \vc{u}_{h} \in_{b} L^2(0,T;\vc{W}^{\Div,2}_{0}(\Omega)).
$$
Consequently, we may assume that there exist functions 
$\vrho,\vc{w},\vc{u}$ such that 
\begin{equation}\label{eq:basic-conv}
	\begin{split}
		& \vrho_{h} \overset{h\to 0}{\weak} \vrho, \quad 
		\text{in $L^\infty(0,T;L^\gamma(\Omega))
		\cap L^{2\gamma}((0,T)\times \Omega)$},\\
		& \vc{w}_{h} \overset{h\to 0}{\weak} \vc{w}, \quad 
		\text{in $L^2(0,T;\vc{W}^{\Curl ,2}_{0}(\Omega))$}, \\
		& \vc{u}_{h} \overset{h\to 0}{\weak} \vc{u}, \quad 
		\text{in $L^2(0,T;\vc{W}^{\Div,2}_{0}(\Omega))$}.
	\end{split}
\end{equation}
Moreover, 
\begin{equation*}
	\vrho_h^\gamma \overset{h\to 0}{\weak}\overline{\vrho^\gamma}, 
	\quad 
	\vrho_h^{\gamma+1} 
	\overset{h\to 0}{\weak}
	\overline{\vrho^{\gamma+1}}, 
	\quad
	\vrho_h\log\vrho_h \overset{h\to 0}{\weak} \overline{\vrho\log\vrho},
\end{equation*}
where each $\overset{h\to 0}\weak$ signifies weak convergence 
in a suitable $L^p$ space with $p>1$.

Finally, $\vrho_h$, $\vrho_h\log\vrho_h$ converge 
respectively to $\vrho$, $\overline{\vrho\log\vrho}$ 
in $C([0,T];L^p_{\text{weak}}(\Om))$ for some 
$1<p<\gamma$, cf.~Lemma \ref{lem:timecompactness} 
and also \cite{Feireisl:2004oe,Lions:1998ga}. 
In particular, $\vrho$, $\vrho\log \vrho$, and 
$\overline{\vrho\log\vrho}$ belong to $C([0,T];L^p_{\text{weak}}(\Om))$.

\subsection{Density scheme} \label{subsec:conv-density}

\begin{lemma}[Convergence of $\vrho_{h} \vc{u}_{h}$]\label{lemma:convergenceofrhou}
Given \eqref{eq:basic-conv},
$$
\vrho_{h}\vc{u}_{h} \overset{h\to 0}{\weak} 
\vrho\vc{u} \quad
\text{in the sense of distributions on $\Dom$.} 
$$
\end{lemma}

\begin{proof}
By virtue of Lemma \ref{lemma:hodge}, there exist 
sequences $\{\vc{\zeta}_{h}\}_{h>0}$, 
$\{\vc{z}_{h}\}_{h>0}$ satisfying
\begin{align*}
	& \vc{u}_{h}(\cdot,t) = \Curl \vc{\zeta}_{h}(\cdot,t)
	+\vc{z}_{h}(\cdot,t),
	\\ & \vc{\zeta}_{h}(\cdot,t) \in \vc{W}_{h}^{0,\perp},
	\quad \vc{z}_{h}(t,\cdot) \in \vc{V}_{h}^{0,\perp},
\end{align*}
for all $t\in (0,T)$. In Lemma \ref{lemma:compactcurlpart} below we prove that 
$$
\Curl \vc{\zeta}_{h} \to \Curl \vc{\zeta} 
\quad \text{in $L^2(0,T;\vc{L}^2(\Omega))$.}
$$
As a consequence, $\Curl \vc{\zeta}_{h}\,\vrho_{h} \weak \Curl \vc{\zeta}\,\vrho$ 
in the sense of distributions. 

It remains to prove that
$$
\vrho_{h}\vc{z}_{h} \weak \vrho \vc{z} 
\quad \text{in the sense of distributions.}
$$
To this end, we adapt the proof of \cite[Lemma 5.1]{Lions:1998ga} to 
our specific discrete setting. We begin by introducing the regularized field 
$\vc{z}_{h}^\epsilon=\kappa^\epsilon \underset{(x)}{\star} \vc{z}_{h}$, 
where $\kappa^\epsilon$ is a standard regularizing kernel 
and $\underset{(x)}{\star}$ denotes the convolution product (in $x$). 
Lemma \ref{lemma:spacetranslation2} guarantees that 
$$
\norm{\vc{z}_{h}^\epsilon-\vc{z}_{h}}_{L^2(0,T;\vc{L}^2(\Omega))} 
\to 0 \quad \textrm{as $\epsilon \to 0$, uniformly in $h$.}
$$
In addition, for any $k$ and $p$, since 
$\vc{z}_{h}^\epsilon \in L^2(0,T;\vc{W}^{k,p}(\Omega))$ 
we have that $\vc{z}_{h}^\epsilon \overset{h\to 0}{\weak} 
\vc{z}^\epsilon$ in $L^2(0,T;\vc{W}^{k,p}(\Omega))$.
Moreover, $\vc{z}^\epsilon \overset{\epsilon\to 0}{\to} \vc{z}$ in 
$L^2(0,T;\vc{L}^2(\Omega))$. Hence, by writing 
$\vrho_{h}\vc{z}_{h} = \vrho_{h}(\vc{z}_{h}-\vc{z}_{h}^\epsilon)
+\vrho_{h}\vc{z}_{h}^\epsilon$ it suffices 
to prove $\vrho_{h}\vc{z}_{h}^\epsilon 
\overset{h\to 0}{\weak} \vrho \vc{z}^\epsilon$ for each fixed $\epsilon>0$.

Next, let us introduce auxiliary functions 
$\vc{Z}^{\epsilon,m}_{h} 
\in \vc{V}_{h}$, $m=1,\ldots,M$, defined by
$$
\vc{Z}^{\epsilon,m}_{h}(x)
=\Delta t \sum_{k=0}^m \vc{z}^{\epsilon,m}_{h}(x), 
\qquad \vc{z}^{\epsilon,m}_{h}=
\kappa^\epsilon \star\vc{z}_{h}^m.
$$
We extend $\{\vc{Z}^{\epsilon,m}_{h}\}_{m=1}^M$ to 
a function $\vc{Z}_{h}^\epsilon$ defined on $(-\Delta t, T]\times \Om$ by setting
$$
\vc{Z}_{h}^\epsilon(t,\cdot) 
= \vc{Z}_{h}^{\epsilon,m}(\cdot), 
\qquad t \in (t^{m-1}, t^m], \quad m=1,\ldots, M,
$$
and $\vc{Z}_h^\epsilon(t,\cdot) = \vc{Z}_h^{\epsilon,0}$, for $t \in (-\Delta t , 0]$.
In view of the regularity of $\vc{z}_{h}^\epsilon$,
\begin{equation}\label{eq:rhou-strongconvwhenintegrated-in-time}
	\vc{Z}_{h}^\epsilon(t,\cdot) \to 
	\vc{Z}^\epsilon(t,\cdot) = \int_{0}^{t}\vc{z}^\epsilon(s,\cdot) \ ds
	\quad \text{in $C^k(\Omega)$ for any $k \geq 0$,}
\end{equation}
uniformly in $t$ on $[0,T]$.

Now, we write
\begin{align*}
	\vrho_{h}^m\vc{z}_{h}^{\epsilon,m}
	=\frac{\vrho_{h}^m \vc{Z}_{h}^{\epsilon,m} 
	- \vrho^{m-1}_{h}\vc{Z}^{\epsilon,m-1}_{h}}{\Delta t}
	- \vc{Z}_{h}^{\epsilon,m-1}
	\frac{\vrho_{h}^m-\vrho^{m-1}_{h}}{\Delta t},
\end{align*}
which alternatively can be written as
$$
\vrho_{h}\vc{z}_{h}^\epsilon 
= \frac{\partial}{\partial t}\Pi_{\mathcal{L}}\left(\vrho_{h}\vc{Z}_{h}^\epsilon\right) 
- \vc{Z}_{h}^\epsilon(\cdot - \Delta t,\cdot)
\frac{\partial}{\partial t} \left(\Pi_{\mathcal{L}}\vrho_{h}\right),
$$
on $(t_{m-1},t_m]\times \Om$, $m=1,\dots,M$.  

Fix $\phi \in C^\infty_{c}(\Dom)$. Summation by parts gives
\begin{align*}
	&\int_{0}^T \int_{\Omega}
	\frac{\partial}{\partial t}\Pi_{\mathcal{L}}
	\left(\vrho_{h}\vc{Z}_{h}^\epsilon\right) \phi\ dxdt 
	\\ & \qquad 
	= -\int_{\Delta t}^T\int_{\Omega} \vrho_{h}(t-\Delta t,x) 
	\vc{Z}_{h}^\epsilon(t-\Delta t,x) 
	\frac{\partial}{\partial t} \left(\Pi_{\mathcal{L}}\phi_{h}\right)\ dxdt,
\end{align*}
where $\phi_{h}(t,\cdot)=\frac{1}{\Delta t}
\int_{t^{m-1}}^{t^m}\phi(s,\cdot) \ ds$ for $t\in (t_{m-1},t_{m})$. 

Thanks to \eqref{eq:basic-conv} and 
\eqref{eq:rhou-strongconvwhenintegrated-in-time}, 
$\vrho_{h}\vc{Z}_{h}^\epsilon \overset{h\to 0}{\weak} 
\vrho \vc{Z}^\epsilon$ in $L^{2\gamma}(0,T;L^{2\gamma}(\Omega))
\cap L^{\infty}(0,T;L^\gamma(\Omega))$, and hence
$$
\frac{\partial}{\partial t}
\Pi_{\mathcal{L}}\left(\vrho_{h}\vc{Z}_{h}^\epsilon\right) 
\overset{h\weak 0}{\weak} 
\frac{\partial}{\partial t}\left(\vrho \vc{Z}^\epsilon\right) 
\quad \text{in the sense of distributions on $\Dom$.}
$$

In addition, Lemma \ref{lemma:timecont} tells us that
$\frac{\partial}{\partial t} \left(\Pi_{\mathcal{L}} \vrho_{h}\right)
\in_{b} L^1(0,T;W^{-1,1}(\Omega))$, and thus
$$
\vc{Z}_{h}^\epsilon(\cdot-\Delta t,\cdot)\frac{\partial}{\partial t} 
\left(\Pi_{\mathcal{L}} \vrho_{h}\right) 
\overset{h\to 0}{\weak} 
\vc{Z}^\epsilon \frac{\partial}{\partial t} \vrho
$$
in the sense of distributions on $\Dom$.

We conclude observing that
$\vrho \vc{z}^\epsilon = \frac{d}{dt}\left(\vrho \vc{Z}^\epsilon\right) 
- \vc{Z}^\epsilon \frac{\partial \vrho}{\partial t}$.
\end{proof}

In the proof of the previous lemma we utilized 

\begin{lemma}\label{lemma:compactcurlpart} 
Given \eqref{eq:basic-conv}, define $\Set{(\vc{\zeta}_{h},\vc{z}_{h})}_{h>0}$ in terms 
of the decomposition $\vc{u}_{h}(t,\cdot) = 
\Curl \vc{\zeta}_{h}(t,\cdot) + \vc{z}_{h}(t,\cdot)$ with 
$\vc{\zeta}_{h}(t,\cdot) \in \vc{W}_{h}^{0,\perp}$, 
$\vc{z}_{h}(t,\cdot) \in \vc{V}_{h}^{0,\perp}$, $t \in (0,T)$. Then
\begin{equation}\label{eq:wh-curlwh-strongconv}
	\vc{w}_{h}  \overset{h\to 0}{\to} \vc{w}, \quad 
	\Curl\vc{\zeta}_{h} \overset{h\to 0}{\to} \Curl \vc{\zeta} 
	\quad \text{in $L^2(0,T;\vc{L}^2(\Omega))$.}
\end{equation}
\end{lemma}

\begin{proof} 
Subtract the first equation of \eqref{FEM:momentumeq} 
with $\vc{v}_h = \Curl \vc{\xi}_h^m$ from $\mu$ times the 
second equation of \eqref{FEM:momentumeq}. 
Multiplying the result with $\Delta t$ and summing 
over all $m=1,\ldots,M$ yields 
\begin{equation}\label{eq:curlconv1}
	\begin{split}
		&\int_0^T \int_\Om \mu \Curl \vc{\eta}_h \Curl \vc{\zeta}_h 
		- \mu \Curl \vc{w}_h \Curl \vc{\xi}_h \ dxdt \\
		&\qquad = \int_0^T\int_\Om \mu \vc{w}_h\vc{\eta}_h
		- \vc{f}_h \Curl \vc{\xi}_h \ dxdt,
	\end{split}
\end{equation}
for all $\vc{\eta}_h, \vc{\xi}_h$ that are piecewise constant 
in time with values in $\vc{W}_h(\Om)$. 
Fixing $\vc{\eta},\vc{\xi} \in C_c^\infty((0,T)\times \Om)$, we use 
in \eqref{eq:curlconv1} the test functions
\begin{align*}
	& \vc{\xi}_h(t,\cdot) = \vc{\xi}_h^m(\cdot) 
	:=\frac{1}{\Delta t}\int_{t^{m-1}}^{t^m} 
	\Pi_h^W \vc{\xi}(\cdot, s) \ ds,
	\quad \text{$t\in (t_{m-1},t_m)$, $m=1,\ldots,M$.} 
	\\ & \vc{\zeta}_h(t,\cdot) =\vc{\zeta}_h^m(\cdot) 
	:=\frac{1}{\Delta t}\int_{t^{m-1}}^{t^m} 
	\Pi_h^W \vc{\zeta}(\cdot, s) \ ds,
	\quad \text{$t\in (t_{m-1},t_m)$, $m=1,\ldots,M$.}
\end{align*}
Due to Lemma \ref{lemma:interpolation}, 
$\Curl \vc{\xi}_h \to \Curl \vc{\xi}$ and 
$\Curl \vc{\eta}_h \to \Curl \vc{\eta}$ in $L^2(0,T;\vc{L}^2(\Omega))$. 
As a consequence, keeping in mind \eqref{eq:basic-conv}, 
we let $h\to 0$ in \eqref{eq:curlconv1} to obtain
\begin{equation}\label{eq:curlconv2}
	\begin{split}
		&\int_0^T \int_\Om \mu \Curl \vc{\eta} \Curl \vc{\zeta}
		-\mu \Curl \vc{w} \Curl \vc{\xi} \ dxdt \\
		&\qquad = \int_0^T\int_\Om \mu \vc{w}\vc{\eta}
		-\vc{f}\Curl \vc{\zeta} \ dxdt,
		\quad \forall \vc{\eta},\vc{\xi} \in \vc{C}_c^\infty((0,T)\times \Om).
	\end{split}
\end{equation}

Since $\vc{C}_c^\infty((0,T)\times \Om)$
is dense in $L^2(0,T;\vc{W}^{\Curl, 2}_0(\Om))$ (\cite{Girault:1986fu}), we 
conclude that \eqref{eq:curlconv2} holds for all 
$\vc{\eta},\vc{\xi} \in L^2(0,T;\vc{W}^{\Curl, 2}_0(\Om))$.
Hence, taking $\vc{\eta} = \vc{w}$, $\vc{\xi}=\vc{\zeta}$ 
in \eqref{eq:curlconv2},
\begin{equation}\label{eq:whatwewant}
	0 = \int_0^T \int_\Om \mu \abs{\vc{w}}^2 
	- \vc{f}\Curl \vc{\xi} \ dxdt.
\end{equation}

Next, setting $\vc{\eta}_h = \vc{w}_h$ and 
$\vc{\xi}_h = \vc{\zeta}_h$ 
in \eqref{eq:curlconv1}, we observe that
\begin{equation*}
	0 = \int_0^T\int_\Om \mu \abs{\vc{w}_h}^2
	-\vc{f}_h\Curl \vc{\xi}_h \ dxdt.
\end{equation*}
Letting $h\to 0$ and comparing the result 
with \eqref{eq:whatwewant} reveals that
$$
\lim_{h \to 0}\int_0^T\int_\Om 
\mu\abs{\vc{w}_h}^2 \ dxdt 
=\int_0^T\int_\Om \mu\abs{\vc{w}}^2 \ dxdt,
$$
which implies the first part of \eqref{eq:wh-curlwh-strongconv}:
\begin{equation}\label{eq:wstrong}
	\vc{w}_h \to \vc{w} \quad 
	\text{in $L^2(0,T;\vc{L}^2(\Omega))$.}
\end{equation}

To prove the second part of \eqref{eq:wh-curlwh-strongconv}, we 
make use of $\vc{\eta}_{h}=\vc{\zeta}_{h}^m$ 
as a test function in the second equation 
of \eqref{FEM:momentumeq}, sum the result over $m=1,\ldots, M$, 
and subsequently send $h$ to zero:
\begin{align*}
	\lim_{h \to 0}
	\int_{0}^T\int_{\Omega}
	\abs{\Curl \vc{\zeta}_{h}}^2\ dxdt
	& = \lim_{h \to 0} \int_{0}^T \int_{\Omega}\vc{w}_{h}\vc{\zeta}_{h}\ dxdt 
	\\ & \overset{\eqref{eq:wstrong}}{=} 
	\int_{0}^T\int_{\Omega}\vc{w}\vc{\zeta}\ dxdt
	= \int_{0}^T \int_{\Omega}|\Curl \vc{\zeta}|^2\ dxdt,
\end{align*}
where the last equality follows by arguing along the 
lines leading up to \eqref{eq:whatwewant}.
\end{proof}

During the proof of Lemma \ref{lemma:convergenceofrhou} we made use of a 
spatial compactness property stated in the next lemma.

\begin{lemma}\label{lemma:spacetranslation2} 
Given \eqref{eq:basic-conv}, define $\Set{(\vc{\zeta}_{h},\vc{z}_{h})}_{h>0}$ in terms
of the decomposition $\vc{u}_{h}(\cdot,t)=\Curl \vc{\zeta}_{h}(t,\cdot)
+\vc{z}_{h}(\cdot,t)$ with $\vc{\zeta}_{h}(\cdot,t) \in \vc{W}_{h}^{0,\perp}$, 
$\vc{z}_{h}(t,\cdot) \in \vc{V}_{h}^{0,\perp}$, for $t \in (0,T)$. 
Then, for any $\xi\in\mathbb{R}^N$,
$$
\norm{\vc{z}_{h}(t,\cdot)-\vc{z}_{h}(t,\cdot-\xi)}_{L^2(0,T;\vc{L}^2(\Omega_\xi)}^2 
\leq C\left( |\xi|^\frac{4-N}{2} + |\xi|^2\right)
\norm{\Div \vc{z}_{h}}_{L^2(0,T;L^2(\Omega))}^2,
$$
where $\Om_\xi = \Set{x \in \Om: \operatorname{dist}(x, \partial \Om)> \xi}$.
\end{lemma}

\begin{proof}
For each $t \in (0,T)$ we know that $\vc{z}_{h}(t,\cdot) 
\in \vc{V}_{h}^{0,\perp}(\Omega)$, so Theorem \ref{theorem:spacetranslation} 
can be applied to give
$$
\norm{\vc{z}_{h} (t,\cdot) - \vc{z}_{h}(t,\cdot - \xi)}_{\vc{L}^2(\Omega_\xi)}^2 
\leq C\left( |\xi|^\frac{4-N}{2} + |\xi|^2\right)
\norm{\Div \vc{z}_{h}}_{L^2(\Omega)}^2,
$$
where $C>0$ is independent of $h,\xi,t$. We conclude by integrating over $(0,T)$.
\end{proof}


\begin{lemma}[Continuity equation]\label{lemma:densityconv}
The limit pair $(\vrho,\vc{u})$ constructed in \eqref{eq:basic-conv} 
is a weak solution of the continuity equation \eqref{eq:contequation} in 
the sense of  Definition \ref{def:weak}.
\end{lemma}

\begin{proof}
Fix a test function $\phi \in C_{c}^\infty([0,T)\times\cOm)$, and 
introduce the piecewise constant approximations 
$\phi_{h}:=\Pi_{h}^Q \phi$, $\phi_{h}^m:= \Pi_{h}^Q \phi^m$, and
$\phi^m:=\frac{1}{\Delta t}\int_{t^{m-1}}^{t^m} \phi(t,\cdot)\ dt$. 

Let us employ $\phi^m_{h}$ as test function in 
the continuity scheme \eqref{FEM:contequation} 
and sum over $m=1,\ldots,M$.  The resulting equation reads
\begin{align*}
	&\sum_{m=1}^M \Delta t\int_{\Omega} \frac{d}{dt}
	\left(\Pi_{\mathcal{L}}\vrho_{h}\right)\phi_{h}^m\ dxdt \\
	& \quad = \sum_{\Gamma \in \Gamma^I_{h}}\sum_{m=1}^M \Delta t
	\int_\Gamma \left(\vrho^m_-(\vc{u}^{m}_h \cdot \nu)^+ 
	+ \vrho^m_+(\vc{u}^{m}_h \cdot \nu)^-\right)
	\jump{\phi^m_{h}}_\Gamma\ dS(x).
\end{align*}
As in the proof of Lemma \ref{lemma:timecont} we can rewrite this as
\begin{equation}\label{eq:contconvs}
	\begin{split}
		&\sum_{m=1}^M \Delta t\int_{\Omega} \frac{d}{dt}
		\left(\Pi_{\mathcal{L}}\vrho_{h}\right)\phi_{h}^m\ dxdt \\
		& = \sum_{m=1}^M \Delta t\int_{\Omega}
		\vrho^m_{h}\vc{u}^m_{h} D\phi^m\ dx
		\\ & \qquad\quad
		+\sum_{E \in E_{h}}\sum_{m=1}^M \Delta t\int_{\binner}
		\jump{\vrho^m_{h}}_{\partial E}(\vc{u}^m_{h} \cdot \nu)^-
		(\phi^m_{h} - \phi^m)\ dS(x)\\
		& = \int_{0}^T \int_{\Omega}\vrho_{h}\vc{u}_{h}D\phi\ dxdt 
		\\ & \qquad\quad
		+ \sum_{E \in E_{h}}\int_{0}^T \int_{\binner}
		\jump{\vrho_{h}}_{\partial E}(\vc{u}_{h}\cdot \nu)^-
		(\phi_{h} - \phi)\ dS(x)dt.
	\end{split}
\end{equation}

Lemma \ref{lemma:productionbound} tells us that
\begin{equation*}
	\abs{\sum_{E \in E_{h}}\int_{0}^T \int_{\binner}
	\jump{\vrho_{h}}_{\partial E}(\vc{u}_{h}\cdot \nu)^-
	(\phi_{h}-\phi)\ dS(x)dt} 
	\leq C\, h^\frac{1}{2}
	\norm{D\phi}_{L^\infty(0,T;\vc{L}^\infty(\Omega))}.
\end{equation*}

In view of Lemma \ref{lemma:convergenceofrhou},
$$
\lim_{h \rightarrow 0 }\int_{0}^T \int_{\Omega}\vrho_{h}\vc{u}_{h} D\phi\ dxdt 
= \int_{0}^T \int_{\Omega}\vrho \vc{u} D\phi\ dxdt.
$$

Summation by parts gives
\begin{equation*}
	\begin{split}
		& \sum_{m=1}^M \Delta t\int_{\Omega} \frac{d}{dt}
		\left(\Pi_{\mathcal{L}}\vrho_{h}\right)\phi_{h}^m\ dxdt
		\\ & \quad 
		= -\int_{\Delta t}^T \int_{\Omega} \vrho_{h}(t-\Delta t,x)
		\frac{\partial}{\partial t} \left(\Pi_{\mathcal{L}} \phi_{h}\right)\ dxdt 
		- \int_{\Omega}\vrho_{h}^0\phi_{h}^1\ dx
		\\ & \quad 
		\overset{h\to 0}{\to}
		-\int_{0}^T\int_{\Omega}\vrho \phi_{t}\ dxdt 
		- \int_{\Omega}\vrho_{0}\phi(0,x)\ dx.
	\end{split}
\end{equation*}
where \eqref{eq:basic-conv}, together with the strong convergence 
 $\vrho^0_h \overset{h \to 0}{\to} \vrho_0$, was used to pass to the limit. 
Summarizing, letting $h \to0$ in \eqref{eq:contconvs} delivers 
the desired result \eqref{eq:weak-rho}
\end{proof}

\subsection{Strong convergence of density approximations}\label{sec:strong-conv-of-vrho}
The instrument used to establish the strong 
convergence of the density approximations $\vrho_h$ is 
a weak continuity property of the quantity $\eff(\vrho_{h},\vc{u}_{h})$ 
defined in \eqref{eq:visc-flux}. To derive this property 
we exploit our choice of numerical method 
and the boundary conditions; specifically, the finite 
element spaces, which are chosen 
such that \eqref{eq:theholepoint} below holds.

\begin{lemma}[Discrete effective viscous flux]\label{lemma:effectiveflux} 
Given the convergences in \eqref{eq:basic-conv},
$$
\lim_{h \to 0}\int_{0}^t\int_{\Omega}
\eff(\vrho_{h},\vc{u}_{h})\, \vrho_{h}\ dxds 
=\int_{0}^t\int_{\Omega} \overline{\eff(\vrho,\vc{u})}\,\vrho\ dxds, 
\quad \forall t\in (0,T).
$$
\end{lemma}

\begin{proof}
For each $m=1,\ldots,M$, consider the problem
\begin{equation}\label{eq:vh-qh}
	\Div \vc{v}^m_{h} = q_h^m - \frac{1}{|\Om|}\int_\Om q_h^0 ~dx, \qquad
	q_h^m:=\vrho_{h}-
	\frac{1}{\Delta t}\int_{t^{m-1}}^{t^m}
	\Pi_{h}^Q \vrho\ dt,
\end{equation}
where $q_h^0 = \vrho^0_h - \vrho_0$.
Observe that $\int_{\Omega}q_{h}^m\ dx=0$. Indeed, using the 
continuity scheme \eqref{FEM:contequation} and 
the continuity equation satisfied by the 
limit $\vrho$, cf.~Lemma \ref{lemma:densityconv},
$$
\int_{\Om} \vrho_{h}\ dx = \int_{\Om} \vrho_h^0\ dx, \qquad
\int_{\Om}\Pi_{h}^Q \vrho\ dx
= \int_{\Om} \vrho\ dx 
= \int_\Om \vrho_0~dx
$$
Thus, there exists a unique solution 
$\vc{v}_{h}^m \in \vc{V}_{h}^{0, \perp}$ of \eqref{eq:vh-qh}. 
We denote by $\vc{v}_{h}(t,\cdot)$,  $q_h(t,\cdot)$ the 
usual ``piecewise constant" extensions of $\left\{v_h^m\right\}_{m=1}^M$, 
$\left\{q_h^m\right\}_{m=1}^M$ to $(0,T)$.

Utilizing $\vc{v}_{h}^m$ as test function, the 
velocity scheme \eqref{FEM:momentumeq} reads
\begin{equation}\label{eq:theholepoint}
	\int_{\Omega}\eff(\vrho_{h}^m, \vc{u}_{h}^m) q_h^m\ dx 
	= \frac{1}{|\Om|}\int_\Om q_h^0~dx\left(\int_\Om \eff(\vrho_h^m, \vc{u}_h^m)~dx\right) - \int_{\Omega}\vc{f}_{h}^m \vc{v}_{h}^m\ dx. 
\end{equation}
Multiplying by $\Delta t$, summing over $m$, and 
using the definition of $q_h^m$, we arrive at
\begin{equation*}
\begin{split}
\int_{0}^t\int_{\Omega}
\eff(\vrho_{h},\vc{u}_{h})
(\vrho_{h} - \vrho)\ dxds 
&= \frac{1}{|\Om|}\int_\Om q_h^0~dx\left(\int_0^t \int_\Om \eff(\vrho_h, \vc{u}_h)~dxdt\right) \\
&\qquad -\int_{0}^t\int_{\Omega}
\vc{f}_{h}\vc{v}_{h}\ dxds,
\end{split}
\end{equation*}
for any $t\in(0,T)$.

In view of Theorem \ref{theorem:spacetranslation}, we have that $\vc{v}_{h} \weak 0$ in $L^2(0,T;\vc{L}^2(\Omega))$. 
Since $\vc{f}_{h} \rightarrow \vc{f}$ in $L^2((0,T)\times \Omega)$ and $\int_\Om q_h^0~dx \rightarrow 0$,
we  conclude the desired result
$$
\lim_{h \rightarrow 0}\int_{0}^t\int_{\Omega}
\eff(\vrho_{h}, \vc{u}_{h})
(\vrho_{h} - \vrho)\ dxds = 0.
$$
\end{proof}

We are now in a position to infer the sought-after strong convergence 
of the density approximations.

\begin{lemma}[Strong convergence of $\vrho_h$]\label{lem:strong-conv}
Suppose that \eqref{eq:basic-conv} holds. Then, passing to 
a subsequence if necessary,
$$
\vrho_{h} \rightarrow \vrho
\quad \text{a.e.~in~$(0,T)\times \Omega$.}
$$

\end{lemma}

\begin{proof}	
In view of Lemma \ref{lemma:densityconv}, the 
limit $(\vrho,\vc{u})$ is a weak solution 
of the continuity equation and hence, by Lemma \ref{lemma:feireisl}, 
also a renormalized solution. In particular, 
$$
\left(\vrho\log \vrho\right)_t 
+ \Div \left( \left(\vrho\log\vrho\right)
\vc{u}\right)=\vrho \Div \vc{u} 
\quad \text{in the weak sense on $\cDom$.}
$$

Since $t\mapsto \vrho\log \vrho$ is continuous 
with values in some Lebesgue space 
equipped with the weak topology, we can use this 
equation to obtain for any $t>0$
\begin{equation}\label{eq:stngdenconv-eq1}
	\int_{\Omega} \left(\vrho \log \vrho\right)(t)\ dx
	-\int_{\Omega}\vrho_{0}\log \vrho_{0}\ dx
	= -\int_{0}^t \int_{\Omega}\vrho \Div \vc{u}\ dxds
\end{equation}

Next, we specify $\phi_h\equiv 1$ as test function in the 
renormalized scheme \eqref{FEM:renormalized}, multiply by $\Delta t$,
and sum the result over $m$. Making use of the 
convexity of $z\log z$, we infer for any $m=1,\dots,M$ 
\begin{equation}\label{eq:stngdenconv-eq2}
	\int_{\Omega}\vrho^m_{h}\log \vrho^m_{h}\ dx
	-\int_{\Omega}\vrho^{0}_h\log \vrho^{0}_h\ dx  
	\leq -\sum_{k=1}^m \Delta t\int_{\Omega}\vrho^m_{h}
	\Div \vc{u}^m_{h}\ dxdt.
\end{equation}

In view of the convergences stated at the beginning of this section 
and strong convergence of the initial data, we 
can send $h \to 0$ in \eqref{eq:stngdenconv-eq2} to obtain
\begin{equation}\label{eq:stngdenconv-eq3}
	\int_{\Omega} \Bigl(\overline{\vrho \log \vrho}\Bigr)(t)\ dx
	-\int_{\Omega}\vrho_{0}\log \vrho_{0}\ dx
	\le -\int_{0}^t \int_{\Omega}\overline{\vrho \Div \vc{u}}\ dxds.
\end{equation}

Subtracting \eqref{eq:stngdenconv-eq1} 
from \eqref{eq:stngdenconv-eq3} gives
\begin{align*}
	\int_{\Omega}\Bigl(\overline{\vrho \log \vrho}-\vrho \log \vrho\Bigr)(t)\ dx
	& \leq -\int_{0}^t\int_{\Omega}
	\overline{\vrho\Div \vc{u}}-\vrho \Div \vc{u}\ dxds,
\end{align*}
for any $t\in (0,T)$. Lemma \ref{lemma:effectiveflux} tells us that
\begin{equation*}
	\int_{0}^t\int_{\Omega} \overline{\vrho \Div \vc{u}}
	-\vrho \Div \vc{u}\ dxds 
	= \frac{a}{\mu + \lambda}\int_{0}^t\int_{\Omega}
	\overline{\vrho^{\gamma +1}}
	-\overline{\vrho^\gamma} \vrho\ dxds\ge 0,
\end{equation*}
where the last inequality follows 
as in  \cite{Feireisl:2004oe,Lions:1998ga}, so 
the following relation holds: 
$$
\overline{\vrho \log \vrho}=\vrho \log \vrho
\quad \text{a.e.~in $\Dom$.}
$$
Now an application of Lemma \ref{lem:prelim} brings the proof to an end.
\end{proof}

\subsection{Velocity scheme}\label{subsec:conv-velocity}

\begin{lemma}[Velocity equation]
The limit triple $(\vc{w},\vc{u},\vrho)$ 
constructed in \eqref{eq:basic-conv} is a weak solution of 
the velocity equation \eqref{eq:momentumeq} in the sense of \eqref{def:mixed-weak}.
\end{lemma}

\begin{proof}
Fix $(\vc{v}, \vc{\eta}) \in \vc{C}^\infty_c((0,T)\times \Om)$, and introduce the projections 
$\vc{v}_{h} = \Pi_{h}^V \vc{v}$, $\vc{\eta}_{h} =\Pi_{h}^W \vc{\eta}$ and 
$\vc{v}_{h}^m = \frac{1}{\Delta t}\int_{t^{m-1}}^{t^m}\vc{v}_{h} \ dt$, 
$\vc{\eta}_{h}^m = \frac{1}{\Delta t}\int_{t^{m-1}}^{t^m}\vc{\eta}_{h} \ dt$. 

Utilizing $\vc{v}^m_{h}$ and $\vc{\eta}^m_{h}$ as test functions in 
the velocity scheme \eqref{FEM:momentumeq}, multiplying by $\Delta t$, 
and summing the result over $m$, we gather
\begin{equation}\label{eq:approx-mixed-weak}
	\begin{split}
		&\int_0^T\int_{\Omega}\mu\Curl \vc{w}_{h}\vc{v}_{h} 
		+ \left[(\mu + \lambda)\Div \vc{u}_{h}
		-p(\vrho_{h})\right]\Div \vc{v}_{h}\ dxdt
		= \int_0^T\int_{\Omega}\vc{f}_{h}\vc{v}_{h}\ dxdt, \\
		&\int_0^T\int_{\Omega}\vc{w}_{h}\vc{\eta}_{h} 
		-\vc{u}_{h}\Curl \vc{\eta}_{h} \ dxdt=0.
	\end{split}
\end{equation}

From Lemma \ref{lemma:interpolation} we have 
$\vc{v}_{h} \overset{h\to 0}{\to} \vc{v}$ in $L^2(0,T;\vc{W}_{0}^{\Div,2}(\Omega))$ 
and $\vc{\eta}_{h}\overset{h\to 0}{\to}\vc{\eta}$ 
in $L^2(0,T;\vc{W}^{\Curl,2}_{0}(\Omega))$. 
Furthermore, by Lemma \ref{lem:strong-conv} and 
the first part of \eqref{eq:basic-conv}, 
$p(\vrho_{h}) \overset{h\to 0}{\to} p(\vrho)$ in $L^2(\Dom)$. 
Hence, we can send $h \to 0$ in \eqref{eq:approx-mixed-weak} 
to obtain that the limit $(\vc{w},\vc{u},\vrho)$ constructed 
in \eqref{eq:basic-conv} satisfies \eqref{def:mixed-weak} for all 
test functions $(\vc{v}, \vc{\eta}) \in \vc{C}^\infty_c((0,T)\times \Om)$.
Since $\vc{C}^\infty_c((0,T)\times \Om)$ is dense 
in both $L^2(0,T;\vc{W}_{0}^{\Div,2}(\Omega))$ and $L^2(0,T;\vc{W}^{\Curl,2}_{0}(\Omega))$
\cite{Girault:1986fu}, this concludes the proof.
\end{proof}


\appendix

\section{Compactness of functions in $V^{0,\perp}_{h}(\Omega)$}
In this appendix we prove that discrete 
weakly curl free approximations in $\vc{V}^{0,\perp}(\Omega)$ 
with $L^2$ bounded divergence possesses an $\vc{L}^2$ 
space translation estimate, which was previously needed to conclude the weak convergence of the product
$\vrho_{h}\vc{u}_{h}$ to the product of the corresponding weak limits $\vrho \vc{u}$.
As part of the proof, in Lemma \ref{lemma:curlcontrol} we show that if 
a sequence $\{\vc{v}_{h}\}_{h>0}$ belongs to $\vc{V}_{h}^{0,\perp}(\Omega)$ and besides satisfies 
$\Div \vc{v}_{h} \in_{b} L^2(\Omega)$, then $\{\Curl \vc{v}_{h}\}_{h>0}$
and $\{\Div \vc{v}_{h}\}_{h>0}$ are actually compact in $W^{-1,2}(\Omega)$.
Thus, strong $L^2(\Omega)$ convergence of a subsequence of $\{\vc{v}_{h}\}_{h>0}$ 
follows directly from the div--curl lemma. However, this is not sufficient 
to conclude the sought after convergence of $\vrho_{h}\vc{u}_{h}$ 
(cf.~Subsection \ref{subsec:conv-density}). 
The problem is a lack of temporal control of the velocity approximations $\vc{u}_{h}$.

\subsection{Space translation estimate}
The argument is inspired by Brenner's work \cite{Brenner:2003fk} 
on Poincar\'e--Friedrich inequalities for piecewise $H^1$ vector fields. 
The basic idea is to project the relevant function into the Crouzeix--Raviart 
element space and then use the standard translation estimate satisfied by 
functions in this space (cf.~ Stummel \cite{Stummel:1980fk}). 
Then, since the relevant function is discrete weakly curl free, we have 
sufficient control on the $\Curl$ to suitably bound the projection error.

The Crouzeix--Raviart element space is defined as a non--conforming 
$\mathbb{P}_{1}$ element for each component of the vector field. 
That is, on each element $E \in E_{h}$ the Crouzeix--Raviart 
polynomial space $\vc{R}(E)$ is given by
$$
\vc{R}(E) = [\mathbb{P}^1(E)]^N,
$$
where $\mathbb{P}_{1}(E)$ is the space of linear scalar 
fields on $E$ and $N$ is the spatial dimension. The degrees of freedom of $\vc{R}(E)$ are 
the average integrals over the faces of $E$. 
The Crouzeix--Raviart element space, denoted $\vc{R}_{h}(\Omega)$, is 
formed on $E_{h}$ by matching the degrees of 
freedom on each face $\Gamma \in \Gamma_{h}$. Hence, $\vc{R}_{h}(\Omega)$ is 
discontinuous across element faces and thus leads 
to non--conforming discretizations of $\vc{H}^{1}(\Omega)$. 
However, it has the property that for any $\vc{v}_{h} \in \vc{R}_{h}$, 
$\int_{\Gamma}[\vc{v}_{h}]\ dS(x) = 0$ for all $\Gamma \in \Gamma_{h}$.

\begin{theorem}\label{theorem:spacetranslation}
Given $\vc{z}_{h}\in \vc{V}^{0,\perp}_{h}$, there exists a 
constant $C>0$, depending only on $\Om$ and the shape regularity of $E_h$, such 
that for every vector $\xi \in \mathbb{R}^N$,
\begin{equation}\label{eq:spacetranslation}
	\norm{\vc{z}_{h}(\cdot)-\vc{z}_{h}(\cdot-\xi)}_{\vc{L}^2(\Om_\xi)} 
	\leq C\left(\abs{\xi}^{\frac{4-N}{2}}+\abs{\xi}^2\right)^\frac{1}{2}
	\norm{\Div \vc{z}_{h}}_{L^2(\Omega)}, 
\end{equation}
where $\Om_\xi=\left\{x\in\Om:\operatorname{dist}(x,\partial\Om)>\xi\right\}$.
\end{theorem}

\begin{proof}
Let us introduce an interpolation operator 
$\Pi_{h}^R: \vc{V}_{h} \rightarrow \vc{R}_{h}$ by specifying
$$
\frac{1}{\abs{\Gamma}}\int_{\Gamma}\Pi_{h}^R \vc{z}_{h}\ dS(x) 
= \frac{1}{\abs{\Gamma}}\int_{\Gamma}\{\vc{z}_{h}\}\ dS(x), 
\qquad \forall \Gamma \in \Gamma^I_{h},
$$
where $\{\cdot\}$ denotes the average of the traces from the two sides of $\Gamma$. 
According to Brenner \cite{Brenner:2003fk}, we have the following error estimate:
\begin{equation}\label{eq:brennersestimate}
	\norm{\vc{z}_{h} - \Pi_{h}^R\vc{z}_{h}}_{\vc{L}^2(\Omega)}^2 
	\leq  C\left(\sum_{E \in E_{h}} h^2
	\norm{D\vc{z}_{h}}_{\vc{L}^2(\Omega)}^2 
	+ \sum_{\Gamma \in \Gamma^I_{h}}h^{2 - N}
	\abs{\int_{\Gamma}[\vc{z}_{h}] \ dS(x)}^2 \right).
\end{equation}
By a standard decomposition of vector fields,
\begin{equation*}
	\jump{\vc{z}_h}_\Gamma 
	= \jump{(\vc{z}_h \cdot \nu)\nu}_\Gamma 
	- \jump{(\vc{z}_h \times \nu)\times \nu}_\Gamma
	= - \jump{(\vc{z}_h \times \nu)\times \nu}_\Gamma,
\end{equation*}
where the last equality follows for the reason that 
$[\vc{z}_{h} \cdot \nu]_{\Gamma}=0$ for all $\Gamma \in \Gamma^I_{h}$. 
To have the above decomposition well--defined in two dimensions, 
 we set $(\vc{z}_h \times \nu)\times \nu = -(\vc{z}_h \times \nu)\tau$, 
where $\tau$ is the tangential vector.

Since $\nu$ is constant on each $\Gamma \in \Gamma_h$ with $|\nu| = 1$,
$$
\abs{\int_{\Gamma}[\vc{z}_{h}] ~dS(x)}^2 
= \abs{\left(\int_{\Gamma}[\vc{z}_{h}\times \nu] \ dS(x)\right)\times \nu}^2
\leq \abs{\int_{\Gamma}[\vc{z}_{h}\times \nu] \ dS(x)}^2, 
\quad \forall \Gamma \in \Gamma^I_h.
$$
As a result, applying \eqref{eq:boundonsumofjumps} of Lemma \ref{lemma:maxbound} below 
to \eqref{eq:brennersestimate} yields the error estimate
\begin{equation}\label{eq:diffOurandCR}
	\norm{\vc{z}_{h} - \Pi_{h}^R\vc{z}_{h}}_{\vc{L}^2(\Omega)}^2 
	\leq C\left(\sum_{E \in E_{h}} h^{2}\norm{D\vc{z}_{h}}_{\vc{L}^2(\Omega)}^2 
	+ h^{\frac{4-N}{2}}
	\norm{\Div \vc{z}_{h}}^2_{L^2(\Omega)}\right).
\end{equation}

To continue, fix an arbitrary $\xi \in \mathbb{R}^N$. 
By the triangle inequality, we write
\begin{equation}\label{eq:translation-addandsubtract1}
	\begin{split}
		&\norm{\vc{z}_{h}(\cdot)-\vc{z}_{h}(\cdot-\xi)}^2_{L^2(\Om_\xi)} \\
		& \leq C\Bigl(\norm{\vc{z}_{h}(\cdot)
		-\Pi_{h}^R \vc{z}_{h}(\cdot)}^2_{L^2(\Om_\xi)}
		\\ &\qquad\qquad 
		+ \norm{\Pi_{h}^R\vc{z}_{h}(\cdot)
		-\Pi_{h}^R\vc{z}_{h}(\cdot-\xi)}^2_{L^2(\Om_\xi)}\\
		& \qquad\qquad\qquad
		+ \norm{\Pi_{h}^R \vc{z}_{h}(\cdot-\xi)
		-\vc{z}_{h}(\cdot-\xi)}^2_{L^2(\Om_\xi)}\Bigr),
	\end{split}
\end{equation}
which transfers the translation onto the projected function $\Pi_h^R \vc{z}$.

Since $\Pi_h^R \vc{z}_h$ is a function in the Crouzeix--Raviart element space, 
Stummel's work \cite[Theorem 2.1]{Stummel:1980fk} can be applied, yielding
\begin{equation}\label{eq:translation-new}
	\begin{split}
		&\norm{\Pi_{h}^R\vc{z}_{h}(\cdot)-\Pi_{h}^R\vc{z}_{h}(\cdot-\xi)}_{\vc{L}^2(\Om_\xi)}^2 
		\\ & \quad
		\leq C\left(h^2 + |\xi|^2\right)\sum_{E\in E_{h}}
		\norm{D\Pi_h^R\vc{z}_{h}}_{\vc{L}^2(E)}^2
		\\ & \quad
		\leq C\left(h^2+|\xi|^2\right)
		\sum_{E\in E_{h}}\norm{D\vc{z}_{h}}_{\vc{L}^2(E)}^2,
	\end{split}
\end{equation}
where the constant $C$ only depends on $\Om$ and the shape regularity of $E_h$. 

Utilizing \eqref{eq:diffOurandCR} and \eqref{eq:translation-new} 
in \eqref{eq:translation-addandsubtract1} gives 
\begin{align*}
	&\norm{\vc{z}_{h}(\cdot)-\vc{z}_{h}(\cdot-\xi)}_{\vc{L}^2(\Om_\xi)}^2 
	\\ & \quad \leq 
	C \left(h^{2}+|\xi|^2\right)
	\sum_{E\in E_{h}}\norm{D\vc{z}_{h}}_{\vc{L}^2(E)}^2 
	+h^{\frac{4-N}{2}}\norm{\Div \vc{z}_{h}}_{L^2(\Om)}^2.		
\end{align*}
Since $\abs{D\vc{z}_h} = \abs{\Div \vc{z}_h}$ on 
each $E \in E_h$, this immediately yields 
\begin{equation}\label{eq:spacetranslation-temp}
	\norm{\vc{z}_{h}(\cdot)-\vc{z}_{h}(\cdot-\xi)}_{\vc{L}^2(\Om_\xi)} 
	\leq C\left(h^{\frac{4-N}{2}}+h^2+\abs{\xi}^2\right)^\frac{1}{2}
	\norm{\Div \vc{z}_{h}}_{L^2(\Omega)}, 
\end{equation}

Finally, let us argue that \eqref{eq:spacetranslation-temp} 
implies \eqref{eq:spacetranslation}. To this end, fix any $\xi \in \mathbb{R}^N$ and $h>0$. 
There exists a shape regular partition $G_{h}$ of $\Omega$ into triangles/tetrahedrals 
such that $\cup_{E \in G_{h}} E = \cup_{E \in E_{h}}E$, 
each $E \in G_{h}$ has a non--empty intersection with at most one element $E \in E_{h}$, and
such that 
$$
\max_{E \in G_{h}}\operatorname{diam}(E) < \frac{|\xi|}{3}.
$$
Next, let $\vc{V}_{|\xi|}(\Omega)$ denote the 
first order div conforming N{\'e}d{\'e}lec element space of first kind formed on the 
mesh $G_{h}$, and let $\Pi_{|\xi|}^V:\vc{V}_{h}(\Omega)\to\vc{V}_{|\xi|}(\Omega)$ 
denote the usual projection into this space.

Now, for any $\vc{z}_{h} \in \vc{V}_{h}^{0,\perp}(\Omega)$, we calculate
\begin{equation}\label{eq:finallyd1}
	\begin{split}
		&\norm{\vc{z}_{h}(\cdot) - \vc{z}_{h}(\cdot-\xi)}_{\vc{L}^2(\Om_\xi)}^2 \\
		& \qquad \leq \norm{\vc{z}_{h}(\cdot) - \Pi_{|\xi|}^V\vc{z}_{h}(\cdot)}_{\vc{L}^2(\Om_\xi)}^2
		+\norm{\Pi_{|\xi|}^V \vc{z}_{h}(\cdot)-\Pi_{|\xi|}^V\vc{z}_{h}(\cdot-\xi)}_{\vc{L}^2(\Om_\xi)}^2 
		\\ &\qquad\qquad 
		+\norm{\Pi_{|\xi|}^V\vc{z}_{h}(\cdot-\xi)-\vc{z}_{h}(\cdot-\xi)}_{\vc{L}^2(\Om_\xi)}^2.
	\end{split}
\end{equation}

By Theorem \ref{theorem:spacetranslation},
\begin{equation}\label{eq:finallyd2}
	\norm{\Pi_{|\xi|}^V \vc{z}_{h}(\cdot)-\Pi_{|\xi|}^V\vc{z}_{h}(\cdot-\xi)}_{\vc{L}^2(\Om_\xi)}^2  
	\leq  C\left(|\xi|^{\frac{4-N}{2}} + \abs{\xi}^2\right)^\frac{1}{2}
	\norm{\Div \vc{z}_{h}}_{L^2(\Omega)}.
\end{equation}
By Lemma \ref{lemma:interpolation},
\begin{equation}\label{eq:finallyd3}
	\norm{\vc{z}_{h}(\cdot)-\Pi_{|\xi|}^V \vc{z}_{h}(\cdot)}_{\vc{L}^2(\Om_\xi)}^2 
	\leq C\abs{\xi}^2 \sum_{E \in E_{h}}
	\norm{D\vc{z}_{h}}_{\vc{L}^2(\Omega)}^2 
	= C\abs{\xi}^2\norm{\Div \vc{z}_{h}}_{L^2(\Om)}^2.
\end{equation}
Inserting \eqref{eq:finallyd2} and \eqref{eq:finallyd3} 
into \eqref{eq:finallyd1} completes the proof of \eqref{eq:spacetranslation}.
\end{proof}

\subsection{Tangential jumps}
In the proof of Theorem \ref{theorem:spacetranslation} 
we harnessed

\begin{lemma}\label{lemma:maxbound}
Given $\vc{z}_{h}\in \vc{V}^{0,\perp}_{h}$, there exists a 
constant $C>0$, independent of $h$, such that
\begin{equation}\label{eq:boundoneachjump}
	\abs{\int_{\Gamma}[\vc{z}_{h} \times \nu]\ dS(x)} 
	\leq C h^{\frac{N}{2}}\norm{\Div \vc{z}_{h}}_{L^2(\Omega)},
	\qquad \forall \Gamma \in \Gamma_h,
\end{equation}
and
\begin{equation}\label{eq:boundonsumofjumps}
	\sum_{\Gamma \in \Gamma^I_{h}}\abs{\int_{\Gamma}[\vc{z}_{h}\times \nu]\ dS(x)}^2 
	\leq C \abs{\Om}^\frac{1}{2} h^{\frac{N}{2}}
	\norm{\Div \vc{z}_{h}}_{L^2(\Omega)}^2.
\end{equation}
\end{lemma}

\begin{proof}
Let $\vc{\phi} \in \vc{W}^{1,2}_0(\Om)$. In virtue 
of Lemma \ref{lemma:curlcontrol} below,
\begin{equation*}
	\abs{\int_{\Omega}\vc{z}_{h}\Curl \vc{\phi}\ dx} 
	\leq C h\norm{\vc{\phi}}_{\vc{W}^{1,2}(\Omega)}
	\norm{\Div \vc{z}_{h}}_{L^2(\Omega)}.
\end{equation*}
Applying integration by parts \eqref{eq:integrationbyparts}, keeping 
in mind that $\Curl \vc{z}_h|_E = 0$ for all $E \in E_h$, yields
$$
\sum_{\Gamma \in \Gamma^I_{h}}
\int_{\Gamma}\vc{\phi}[\vc{z}_{h} \times \nu]\ dS(x)
=\int_{\Omega}\vc{z}_{h}\Curl \vc{\phi}\ dx,
$$
and so
\begin{equation}\label{eq:basiceq}
	\abs{\sum_{\Gamma \in \Gamma^I_{h}}
	\int_{\Gamma}\vc{\phi}[\vc{z}_{h} \times \nu]\ dS(x)} 
	\leq C h\norm{\vc{\phi}}_{\vc{W}^{1,2}(\Omega)}
	\norm{\Div \vc{z}_{h}}_{L^2(\Omega)}.
\end{equation}

The bound \eqref{eq:basiceq} serves as the starting point for proving 
\eqref{eq:boundoneachjump} and \eqref{eq:boundonsumofjumps}; the remaining objective is to 
construct a suitable test function $\vc{\phi}$. Fix $\Gamma \in \Gamma^I_{h}$. 
Let $E_{-}$, $E_{+}$ denote 
the two elements in $E_{h}$ sharing the egde/face $\Gamma$, where $E_{-}$, $E_{+}$ 
are chosen so that $\nu$ points from $E_{-}$ to $E_{+}$.
In view of Lemma \ref{lem:testfunc}, we can choose a continuous piecewise linear 
(scalar) function $\tilde{\phi}$ on $\Gamma$ such that
$$
\int_{\Gamma}\tilde{\phi}f\ dx =\frac{1}{N}\int_{\Gamma}f\ dx, 
\qquad \forall f \in \mathbb{P}_{1}(\Gamma).
$$
Denote by $\phi_{\partial E}$ the extension by zero of $\tilde{\phi}$ to 
$\left(\partial E_{-}\setminus\Gamma\right)\bigcup 
\left(\partial E_+\setminus\Gamma\right)$, and fix a 
piecewise affine function $\phi_{E}$ on $E_{-}\cup E_{+}$ such that 
$\phi_{E}\big|_{\partial E_{+} \cup \partial E_-} = \phi_{\partial E}$. 
Clearly, $\phi_E$ can be chosen such that
$$
\text{$\abs{D\phi_{E}} \leq C h^{-1}$ 
in the interior of $E_{-}\cup E_{+}$.}
$$
Finally, let $\phi_\Gamma$ denote the extension 
by zero of $\phi_E$ to all of $\Om$.

The function $\phi_\Gamma$ possesses the following properties: 
$\phi_\Gamma \in W^{1,2}_{0}(\Omega)$, $\phi_\Gamma\big|_{\tilde\Gamma}=0$ 
for all $\tilde \Gamma \in \Gamma_{h}$ such that $\tilde\Gamma \neq \Gamma$, and
$$
\norm{\phi_\Gamma}^2_{W^{1,2}(\Omega)} 
= \norm{\phi_\Gamma\big|_{E_{-}}}^2_{W^{1,2}(E_{-})} 
+\norm{\phi_\Gamma\big|_{E_{+}}}^2_{W^{1,2}(E_{+})} 
\leq Ch^{N}\left(1+h^{-2}\right).
$$

If $N=2$ (curl is scalar), then we opt for $\vc{\phi}=\phi_\Gamma$ 
in \eqref{eq:basiceq} to obtain 
\begin{equation}\label{eq:firstpartproved}
	\abs{\int_{\Gamma} [\vc{z}_{h} \times \nu]\ dS(x)} 
	\leq C h^{\frac{N}{2}}\norm{\Div \vc{z}_{h}}_{L^2(\Om)}.
\end{equation}
If $N=3$, \eqref{eq:firstpartproved} still holds. Indeed, to conclude 
we can in \eqref{eq:basiceq} successively take $\vc{\phi}= [\phi_\Gamma,0,0]^T$, 
$\vc{\phi}= [0,\phi_\Gamma,0]^T$, and $\vc{\phi}= [0,0,\phi_\Gamma]^T$. 
Since $\Gamma \in \Gamma_{h}$ was arbitrary, this 
concludes the proof of \eqref{eq:boundoneachjump}.

To establish \eqref{eq:boundonsumofjumps}, we introduce the test function
$$
\vc{\phi}_{\Gamma}= 
\phi_\Gamma \int_\Gamma \jump{\vc{z}_h \times \nu} \ dS(x),
\qquad \forall \Gamma \in \Gamma^I_{h},
$$
where $\phi_\Gamma$ is constructed as above 
with the additional requirement that
$$
\operatorname{supp}\,\phi_{\Gamma} \bigcap 
\operatorname{supp}\,\phi_{\tilde\Gamma}=\emptyset, 
\qquad \forall \tilde\Gamma \neq \Gamma.
$$
We have 
$$
\sup_{x \in \Om} \abs{D\vc{\phi}_{\Gamma}}
\leq C h^{-1}\abs{\int_{\Gamma}[\vc{z}_{h} \times \nu]\ dS(x)},
\qquad \forall \Gamma \in \Gamma^I_h,
$$
and, for each $\Gamma\in\Gamma_h$,
$$
\int_\Gamma \vc{\phi}_\Gamma f \ dS(x)
= \frac{1}{N}\left(\int_\Gamma \jump{\vc{z}_h \times \nu}~dS(x)\right)
\left(\int_\Gamma f~dS(x)\right), \qquad 
\forall f\in \mathbb{P}^1(\Gamma).
$$

Finally, we set $\vc{\phi}:=\sum_{\Gamma \in \Gamma^I_{h}}\vc{\phi}_{\Gamma}$; this 
function satisfies $\vc{\phi}\in\vc{W}^{1,2}_{0}(\Omega)$ and
$$
\sup_{x \in \Omega} \abs{D\vc{\phi}} 
\leq C h^{-1}\max_{\Gamma \in \Gamma^I_{h}}
\abs{\int_{\Gamma}[\vc{z}_{h}\times \nu]\ dS(x)} 
\leq  C h^\frac{N-2}{2}
\norm{\Div \vc{z}_{h}}_{L^2(\Omega)}.
$$
The last inequality follows from \eqref{eq:boundoneachjump}. 
A direct calculation gives
$$
\norm{D\vc{\phi}}_{\vc{L}^2(\Omega)} 
\leq C h^\frac{N-2}{2}
\abs{\Omega}^\frac{1}{2}
\norm{\Div \vc{z}_{h}}_{L^2(\Omega)}.
$$

Setting $\vc{\phi}$ as test function in \eqref{eq:basiceq} 
immediately gives the estimate
\begin{equation*}
	\sum_{\Gamma \in \Gamma_h^I}\abs{\int_{\Gamma} [\vc{z}_{h} \times \nu]\ dS(x)}^2 
	\leq C h^{\frac{N}{2}}\abs{\Omega}^\frac{1}{2}
	\norm{\Div \vc{z}_{h}}^2_{L^2(\Omega)},
\end{equation*}
which is \eqref{eq:boundonsumofjumps}.
\end{proof}

The next lemma provides us with the specific 
test function that was brought into service in the above proof.

\begin{lemma}\label{lem:testfunc}
Fix any $\Gamma \in \Gamma_h$. There exists a continuous 
piecewise linear (scalar) function $\phi$ on $\Gamma$ such that 
$\phi|_{\partial \Gamma} = 0$, $|\phi(x)| \leq 1$ $\forall x \in \Gamma$, and
\begin{equation}\label{eq:phireq}
	\int_\Gamma \phi f \ dx = 
	\frac{1}{N}\int_\Gamma f\ dx, 
	\qquad \forall f \in \mathbb{P}^1(\Gamma),
\end{equation}
where $N$ is the spatial dimension.
\end{lemma}

\begin{proof}
Let $b$ denote the barycentric middle point with respect 
to the vertices of $\Gamma$. Let $T_h$ be the triangulation 
of $\Gamma$ obtained by setting $b$ as a vertex in addition
to the vertices of $\Gamma$. On $T_h$ let $L_h(\Gamma)$ denote the standard 
finite element space of continuous piecewise linear functions. Any 
function $\phi_h \in L_h(\Gamma)$ is 
uniquely determined by it's value at the vertices. 

Now the relevant test function $\phi \in L_h(\Gamma)$ is obtained 
by requiring
$$
\text{$\phi(b)= 1$ and $\phi(v_i)=0$, for all vertices $v_i$ at $\pOm$.}
$$
By direct calculation it can be verified that $\phi$ satisfies \eqref{eq:phireq}.
\end{proof}

\subsection{Negative space compactness of the curl}
In the proof of Lemma \ref{lemma:maxbound}, the essential ingredient 
was an estimate on the $W^{-1,2}$ norm of $\Curl \vc{z}_{h}$. 
In this subsection, we prove this result.

\begin{theorem}
\label{theorem:mixedelliptic}
Consider the mixed Laplace-type problem
\begin{equation}\label{eq:hodgelaplace}
	\begin{split}
		& \Curl \vc{w} - D\Div \vc{u} = \vc{f}, 
		\, \vc{w} = \Curl \vc{u} \quad \text{in $\Om$,} \\
		& \vc{u}\cdot \nu = 0, \, \vc{w} \times \nu = 0 
		\quad \text{on $\pOm$,} \\
	\end{split}
\end{equation}
where we assume $\vc{f} \in \vc{L}^2(\Omega)$. There exists a pair
$$
(\vc{w},\vc{u}) \in \vc{W}_{0}^{\Curl,2}(\Omega)\times \vc{W}_{0}^{\Div, 2}(\Omega), 
$$ 
satisfying \eqref{eq:hodgelaplace} in the weak sense. 
Moreover, there exists a pair
$$
(\vc{w}_h,\vc{u}_h)\in \vc{W}_{h}(\Omega) \times  \vc{V}_{h}(\Omega), \qquad 
$$
satisfying the corresponding mixed finite element formulation of \eqref{eq:hodgelaplace}.
Finally, the following error estimate holds:
\begin{equation}\label{eq:mixedellipticregularity}
	\begin{split}
		\norm{\vc{w} - \vc{w}_{h}}_{\vc{L}^2(\Omega)}
		+\norm{\vc{u} - \vc{u}_{h}}_{\vc{V}_{h}}
		& \leq Ch^s \|\vc{f}\|_{\vc{L}^2(\Omega)}, \\
	\end{split}
\end{equation}
where the convergence rate $s\in [1/2,1)$ depends on the regularity of $\partial \Omega$. 
If $\partial \Omega$ is Lipschitz and convex, \eqref{eq:mixedellipticregularity} holds with $s=1$.
\end{theorem}
\begin{proof}
For example, cf.~Theorem 7.9 in \cite{Arnold:2006wj}.
\end{proof}

\begin{lemma}
\label{lemma:curlcontrol}
Let $\{\vc{z}_{h}\}_{h>0}$ be a sequence in $\vc{V}^{0,\perp}_{h}$ for which 
$\norm{\Div \vc{z}_{h}}_{L^2(\Omega)}\leq C$, where 
the constant $C>0$ is independent of $h$. Then
\begin{equation*}
	\norm{\Curl \vc{z}_{h}}_{\vc{W}^{-1,2}(\Omega)}  
	\leq Ch \norm{\Div \vc{z}_{h}}_{L^2(\Omega)},
\end{equation*}
for some constant $C$ independent of $h$.
\end{lemma}

\begin{proof}
To prove this lemma, we will use the mixed system \eqref{eq:hodgelaplace} to 
define a new operator. To motivate the construction, consider the problem
\begin{equation}\label{eq:beginning}
	\begin{split}
		& -\Delta \vc{\theta} = \Curl \vc{\phi} \quad \text{in $\Om$}, 
		\quad \vc{\theta} \cdot \nu = 0,\, 
		\Curl \vc{\theta} \times \nu = 0
		\quad \text{on $\pOm$},
	\end{split}
\end{equation}	
for some given $\vc{\phi} \in \vc{W}^{\Curl,2}_0(\Om)$. 
By utilizing $D \Delta^{-1} \Div \vc{\theta}$
as test function in the weak formulation of \eqref{eq:beginning}, where $\Delta^{-1}$ 
is the Neuman Laplace inverse, it is easily seen that the weak solution 
$\vc{\theta}$ of the system \eqref{eq:beginning} is divergence free. 
Furthermore, we can set $\vc{w} = \Curl \vc{\theta}$ and integrate 
by parts to conclude that the pair $(\vc{w}, \vc{\theta})$ is also 
the unique weak solution of the mixed Laplace 
system \eqref{eq:hodgelaplace} with $\vc{f} = \Curl \vc{\phi}$.

Now we define a new operator $\Pi^h: \vc{W}^{\Curl,2}_{0} \to \vc{W}_{h}$ 
as the unique function $\Pi^h\vc{\phi} \in \vc{W}_{h}$ 
satisfying the finite element formulation:
\begin{equation}\label{eq:operatorformulation}
	\begin{split}
		\int_{\Omega}\Curl (\Pi^h\vc{\phi})\vc{v}_{h}
		+\Div \vc{\theta}_{h}\Div \vc{v}_{h}\ dx 
		& =\int_{\Omega}\Curl \vc{\phi} \vc{v}_{h}\ dx, 
		\quad \forall \vc{v}_{h} \in \vc{V}_{h},\\
		\int_{\Omega}\Curl \vc{\eta}_{h}\vc{\theta}_{h}\ dx 
		& =\int_{\Omega}(\Pi^h\vc{\phi})\vc{\eta}_{h}\ dx, 
		\quad \forall \vc{\eta}_{h} \in \vc{W}_{h}.
	\end{split}
\end{equation}
The existence of such a function $\Pi^h\vc{\phi}$ is 
given by Theorem \ref{theorem:mixedelliptic}. 
Using the fact that $\Div \vc{\theta}= 0$, the error 
estimate \eqref{eq:mixedellipticregularity} yields
\begin{equation}\label{eq:projerror}
	\norm{\Pi^h \vc{\phi}-\Curl \vc{\theta}}_{\vc{L}^2(\Om)} 
	+\norm{\vc{\theta}_h-\vc{\theta}}_{\vc{L}^2(\Om)} 
	+\norm{\Div \vc{\theta}_h}_{L^2(\Om)}
	\leq Ch^s \norm{\Curl \vc{\phi}}_{\vc{L}^2(\Om)}.
\end{equation}

Let $\vc{z}_{h}$ be as stated in the lemma. Since $\vc{z}_{h}$ is 
orthogonal to functions in $\vc{W}_{h}$,
\begin{equation*}
	\begin{split}
		\norm{\Curl \vc{z}_{h}}_{\vc{W}^{-1,2}(\Om)}
		& = \sup_{\vc{\phi}\in\vc{W}^{1,2}_{0}}
		\frac{\abs{\int_{\Om}\vc{z}_{h}\Curl\vc{\phi}\ dx}}{\norm{\vc{\phi}}_{\vc{W}^{1,2}(\Om)}}
		\\ & = \sup_{\vc{\phi}\in\vc{W}^{1,2}_{0}}\frac{\abs{\int_{\Om}\vc{z}_{h}
		\Curl (\vc{\phi}-\Pi^h\vc{\phi})\ dx}}{\norm{\vc{\phi}}_{\vc{W}^{1,2}(\Om)}}
		\\ & \overset{\eqref{eq:operatorformulation}}{=} 
		\sup_{\vc{\phi} \in \vc{W}^{1,2}_{0}}
		\frac{\abs{\int_{\Om}\Div \vc{z}_{h}
		\Div \vc{\theta}_{h}\ dx}}{\norm{\vc{\phi}}_{\vc{W}^{1,2}(\Om)}}
		\\ & \le C h \norm{\Div \vc{z}_{h}}_{L^2(\Om)},
	\end{split}
\end{equation*}
where we have used \eqref{eq:projerror} to derive the last 
inequality, specifically the estimate
$$
\norm{\Div \vc{\theta}_{h}}_{L^2(\Om)} 
\leq C h\norm{\Curl \vc{\phi}}_{L^2(\Om)}
\le Ch\norm{\vc{\phi}}_{\vc{W}^{1,2}(\Om)}.
$$
This concludes the proof.
\end{proof}

\end{document}